\documentclass[reqno]{amsart}
\usepackage{latexsym,amsfonts,amssymb,amsmath,euscript}
\usepackage{amssymb,amsmath,amsthm,amsfonts,amscd}
\usepackage{physics}
\usepackage[dvipsnames]{xcolor}% tem mais cores que o acima.
\usepackage[colorlinks=true,citecolor=blue,linkcolor=blue]{hyperref}
\usepackage[round,authoryear]{natbib}

\usepackage{graphicx} % inserir imagens
\usepackage{color,fancybox}
\usepackage{graphicx} % inserir imagens
\usepackage{hyperref}         % para acrescentar links

\numberwithin{equation}{section}
\newtheorem{theorem}{Theorem}[section]
\newtheorem{lemma}[theorem]{Lemma}

\newtheorem{proposition}[theorem]{Proposition}
\newtheorem{remark}[theorem]{Remark}

\newtheorem{hypothesis}{Hypothesis}

\allowdisplaybreaks
\usepackage{esint}
\usepackage{enumerate}

\title[Parabolic--Elliptic Dynamics with Local--Nonlocal Coupled Operators]{Parabolic--Elliptic Dynamics with Local--Nonlocal Coupled Operators}

\begin{document}

\author[L.C. Rosa da Silva]
{Luiza C. Rosa da Silva}

\author[J. D. Rossi]
{Julio D. Rossi}

\address{Luiza Camile Rosa da Silva 
\hfill\break\indent Depto. de Matem\'atica, Instituto de Matem\'atica e Estat\'istica 
\hfill\break\indent Universidade de S\~ao Paulo, Rua do Mat\~ao 1010, S\~ao Paulo - SP - Brazil}
\email{luiza@ime.usp.br}

\address{Julio Daniel Rossi
\hfill\break\indent Depto. de Matematicas y Estadistica, Universidad Torcuatto di Tella
\hfill\break\indent Av. Pres. Figueroa Alcorta 7350, C1428, Buenos Aires - Argentina}
\email{julio.rossi@utdt.edu}

\date{}



\begin{abstract}
In this paper, we study two local--nonlocal settings for parabolic--elliptic evolution systems. 
In our problems we have a disjoint partition of the spacial domain $\Omega$ as $\Omega=A\cup B$
and we first consider a local parabolic equation posed in  
$A$ with a nonlocal elliptic balance equation acting in the complementary subdomain $B$. 
Next, we reverse the roles and take a local elliptic equation posed in  
$A$ coupled with a nonlocal parabolic equation acting in $B$.
In both models, the interaction between the two regions is driven by a nonlocal transmission term given by a kernel that transfers mass across the interface, giving rise to a mixed local--nonlocal, elliptic--parabolic dynamics.
We consider Neumann boundary conditions for both problems.
To being our analysis we first establish the existence and uniqueness of solutions using a fixed
point argument. Then, we provide a detailed analysis of their qualitative behavior. In particular, we show that the coupling structure induces a natural energy functional whose gradient flow governs the evolution, despite the elliptic--parabolic nature of the system. As it is expected in Neumann settings, we prove that the total mass in the whole domain $\Omega$ is preserved in time.
We also analyze the long-time behaviour and obtain decay estimates for the parabolic component, which in turn drive the convergence of the elliptic part to a constant solution. Finally, we prove that the parabolic--elliptic problem under consideration is
the limit of a purely parabolic problem when a parameter that controls the speed of the dynamic at which one component evolves goes to zero.
\end{abstract}

\maketitle
\allowdisplaybreaks
\tableofcontents

\section{Introduction and statements of the main results}

Our main goal is to study two evolution problems that couples a local (nonlocal) parabolic equation in some part of the domain with a nonlocal (resp. local) elliptic equation in the remaining part. Both equations interact through a nonlocal transmission condition. 
The problems that we consider here do not allow mass transfer across the boundary of the domain. 
For this reason, both models can be interpreted as having a Neumann-type boundary
conditions; see \cite{julio, acosta} for further details. For an alternative formulation of Neumann conditions in the nonlocal setting, we refer the reader to \cite{toninho}. 
Other boundary conditions are possible (for example, we may prescribe Dirichlet boundary data, or even mixed boundary conditions),
but to keep the length of this paper reasonable we restrict ourselves to the Neumann case.

Nonlocal models can describe phenomena
not well represented by classical local Partial Differential Equations, PDE.
We refer to problems characterized by long-range interactions and discontinuities. 
For instance, in the context of diffusion, long-range interactions effectively
describe anomalous diffusion, while in the context of mechanics, material discontinuities
appear in crack formation. 
The fundamental difference between nonlocal models and classical local models  
is the fact that the latter involve differential operators, whereas the
former rely on integral operators. For general references on nonlocal models 
with applications to elasticity, population dynamics, image processing, we refer
to \citet{acosta,BatesChmaj99,BatesFifeRenWang97,BrandleChasseigneQuiros12,
CanizoMolino18,CarrilloFife05,ChasseigneChavesRossi06,
CortazarElguetaRossiWolanski07,CortazarElguetaRossiWolanski08,cortazar,
Coville,DiPaola,Fife03,Hutson,MenDu} and the book \citet{julio}.

It is often the case that nonlocal effects are concentrated only in
some parts of the domain, whereas, in the remaining parts, the system can be accurately
described by a PDE. The goal of coupling local and nonlocal models is to
combine a local equation (a PDE) with a nonlocal one (an integral equation),
under the assumption that the location of local and nonlocal effects can be identified
in advance. 
In this
context, one of the challenges of a coupling strategy is to provide a mathematically 
consistent formulation.
 
From a mathematical point of view, interesting properties arise from coupling local and nonlocal models, 
we refer to \cite{Peri1,Peri2,Peri3,monia,DEliaPeregoBochevLittlewood16,DuLiLuTian18,delia3,GalWarma17,Han,Kriventsov15,Sel,Sel2,Sel3}
the survey \cite{SUR} and references therein. 
Previous strategies treat the coupling condition as an optimization objective (the goal is to minimize the mismatch of the local and nonlocal solutions inside a common region). In \cite{Bere} the effects of network transportation on enhancing biological invasion are studied. The proposed mathematical model consists of one equation with nonlocal diffusion in a one-dimensional domain coupled via the boundary condition with a standard reaction-diffusion, in a two-dimensional domain. 
In \cite{DEliaPeregoBochevLittlewood16}, local and nonlocal problems were coupled through a prescribed region in which both kinds of equations overlap (the value of the solution in the nonlocal part of the domain is used as a Dirichlet boundary condition for the local part and vice-versa). This kind of coupling gives continuity of the solution in the overlapping region but does not preserve the total mass when Neumann boundary conditions are imposed.
In \cite{DEliaPeregoBochevLittlewood16} and \cite{DuLiLuTian18}, numerical schemes using local and nonlocal equations were developed and used to improve the computational
accuracy. In \cite{quiros} and \cite{monia} (see also \cite{Kriventsov15} for regularity results), 
evolution problems related to energies closely related to ours are studied.

Now, let us describe our coupled models. Let us fix a bounded smooth domain $\Omega\subset \mathbb{R}^N$ and consider a partition of the domain into two disjoint sets
\[ \Omega = A \cup B. \]
To simplify some of the arguments, we will also assume that $A$ and $B$ are connected (but this condition can be relaxed as in \cite{acosta}). 

\subsection{First local-nonlocal model}
In our first model, we consider a diffusion process acting on the whole domain $\Omega$, with particles diffusing much faster inside $B$ than inside $A$. This leads naturally to a coupled model in which the dynamics in $A$ are governed by a local parabolic equation (with the classical heat equation as main part), while in $B$ they evolve through a nonlocal elliptic operator. The interaction between the two regions is given by a nonlocal transmission term.

Let us describe precisely our system. We look for a pair of functions $(u,v)$,
$u:A\times [0,\infty) \mapsto \mathbb{R}$ and  $v:B\times [0,\infty) \mapsto \mathbb{R}$, that verify the following set of equations. In the first region, $A$, we consider a diffusion problem with homogeneous Neumann boundary conditions
\begin{equation}\label{localpar}
\begin{cases}
    \displaystyle u_t(x,t)=\Delta u(x,t)+ \int_B J(x-y)(v(y,t)-u(x,t))\,dy, &  (x,t)\in A\times (0,\infty),  \\
    \displaystyle \frac{\partial u}{\partial \eta}(x,t)=0, & (x,t)\in \partial A \times (0,\infty),\\
    u(x,0)=u_0(x), & x\in A,
\end{cases}
\end{equation}
where $\eta$ denotes the outward unit normal vector to $\partial A$. In $B$, we look at an elliptic nonlocal equation,
\begin{equation}\label{nleli}
   0 = \int_B G(x-y)(v(y,t)-v(x,t))\,dy
        +\int_A J(x-y)(u(y,t)-v(x,t))\,dy,
   \qquad (x,t)\in B\times (0,\infty).
\end{equation}
Observe that, in the model, the main part of the first equation \eqref{localpar} is simply the local heat equation $u_t=\Delta u$, while in \eqref{nleli} the dominant term corresponds to the elliptic version of a nonlocal diffusion operator, namely
\[
\int_B G(x-y)(v(y,t)-v(x,t))\,dy,
\]
see \cite{julio}. The terms
\[
\int_B J(x-y)(v(y,t)-u(x,t))\,dy
\quad\text{and}\quad
\int_A J(x-y)(u(y,t)-v(x,t))\,dy
\]
provide the coupling between both regions and act as nonlocal transmission conditions. 
For previous references on coupled local/nonlocal equations, we refer the reader to \cite{acosta,monia,
DEliaPeregoBochevLittlewood16,bruna,DuLiLuTian18,GalWarma17,quiros,Kriventsov15}.

This model can also be interpreted from the point of view of a particle system. As mentioned above, we decompose the domain $\Omega$ into two subdomains $A$ and $B$. In $A$, particles move according to Brownian motion (leading to the local heat equation $u_t=\Delta u$) and reflect at the boundary of $A$ (giving the Neumann condition $\frac{\partial u}{\partial \eta}=0$). In contrast, in $B$ particles follow a pure jump process with jump probability given by $G(x-y)$, 
this produces a term of the form
\[
\int_B G(x-y)(v(y,t)-v(x,t))\,dy.
\]
We assume that particles in $B$ jump on a much faster time scale, this makes the equation in $B$ to be elliptic.  
In addition to these internal mechanisms inside $A$ and $B$, particles may jump across the interface with probability density given by $J(x-y)$, which generates the coupling terms in both equations.

It is worth noting that we do not impose continuity of the densities across the interface between $A$ and $B$. However, the densities will remain continuous inside each subdomain due to the regularity of the kernels when we assume the continuity of the initial data. 

Each of the two kernels, $J,\,G$, is a radial probability density, precisely, we assume the following conditions on the kernels. 

\begin{hypothesis}\label{hypo}
The kernel $K = J,\,G \in C(\mathbb{R}^N,\mathbb{R})$ is a nonnegative, continuous and radial function such that 
$K(0) > 0$, $dist(A,B) < \operatorname{supp} K$\footnote{We denote $dist(A,B):=\inf\{|x-y|:\,x\in A,\ y\in B\}$ and 
$\operatorname{supp}K:=\overline{\{z\in\mathbb{R}^N:\,K(z)\neq0\}}$. 
The condition means that there exist $x\in A$ and $y\in B$ such that $x-y\in \operatorname{supp}K$.}, and
\[
\int_{\mathbb{R}^N} K(z)\,dz = 1.
\]
\end{hypothesis}

An important remark is that, since $dist(A,B) <  \operatorname{supp} J$, points inside $A$ interact
with points inside $B$, through the kernel $J$. 
This prevents a degeneracy that would occur if $A$ and $B$ were placed too far apart. 
Indeed, if the distance between $A$ and $B$ were larger than the radius of $\operatorname{supp} J$, then the convolution terms connecting the two regions would vanish, making the system trivial. 
This condition rules out this situation and guaranties that nonlocal interactions remain effective throughout the domain.

A natural approach to understanding the dynamics of the coupled problem \eqref{localpar}--\eqref{nleli} is to examine the energy structure of the system. This leads naturally to the functional $E_v : H^1(A) \to \mathbb{R}$ defined by
\begin{align}\label{Eintro}
E_v(u)
= \frac12\int_A |\nabla u(x)|^2\,dx
  + \frac12\int_A\int_B J(x-y)\,(u(x))^2\,dy\,dx
  - \int_A\int_B J(x-y)\,u(x)v(y)\,dy\,dx,\end{align}
where \(v \in L^2(B)\) is the minimizer of the functional \(F : L^2(B)\to \mathbb{R}\),
\begin{align}\label{Fintro}
F(v)
= \frac14 \iint_{B\times B} G(x-y)\,(v(y)-v(x))^2\,dx\,dy
  + \frac12\int_A\int_B J(x-y)\,(v(y)-u(x))^2\,dy\,dx,
\end{align}
We adopt the convention that $E_v(u)=\infty$ whenever $u \notin H^1(A)$. 
It is straightforward to verify that these energy functionals are proper, convex, lower semicontinuous
and coercive by Lemma \ref{lemma:coerc_v}. Equation \eqref{localpar} is the \(L^2(A)\)-gradient flow of \(E_v\), while equation \eqref{nleli} is the Euler--Lagrange equation associated with the minimization of \(F\). 

With the help of this energy functional we can prove existence and uniqueness of solutions. 

\begin{theorem} \label{teo.existencia.uni.intro}
Given $u_0\in L^2(A)$, there exists a unique fixed point solution $u_{fp}$ to the Neumann problem \eqref{localpar}--\eqref{nleli}. A comparison principle holds for this problem, if $u_0\ge \widetilde{u}_0$ then the corresponding solutions are also ordered,
and, moreover, the total mass is conserved along the evolution.
\end{theorem}

Once we have existence and uniqueness of solutions we turn our attention to 
the asymptotic behavior. 
It is well known that solutions of the Neumann local heat equation 
$u_t = \Delta u $
days exponentially in time to the mean value of the initial condition. 
Our model reproduces this behavior, at least when $A$ is connected. 
The result reads as follows:

\begin{theorem}\label{thm:expdecay.intro} Let $A$ be a smooth connected subdomain and
 $u_0\in L^2(A)$ with $\int_A u_0\,dx = 0$, and let $u$ be the 
solution to \eqref{localpar}.  
Then $u$ decays exponentially to zero as $t\to\infty$.  
Let
\begin{align}\label{introeigenvalue}
\begin{array}{l}
\displaystyle
\lambda_1  = 
\min_{\| u \|_{L^2 (A)=1,\int_A u (x) \, dx =0}} 
 \frac12\int_A |\nabla u(x)|^2\,dx
  + \frac14 \int_B\!\int_B G(x-y)(v(x)-v(y))^2\,dx\,dy 
  \\[10pt]
  \qquad \qquad \qquad \qquad \displaystyle + \frac12\int_A\int_B J(x-y)\,(v(y)-u(x))^2\,dy\,dx.
  \end{array}
\end{align}
Then, it holds that
$$\lambda_1  >0$$
and we have an exponential decay
\begin{equation}\label{abneumann.intro}
\|u(\cdot,t)\|_{L^2(A)}
\;\le\;
e^{- \lambda_1 t}\,\|u_0\|_{L^2(A)},
\qquad t\ge 0.
\end{equation}
\end{theorem}

Notice that, by linearity, this results implies that, for $u_0 \in L^2 (A)$, 
\[
\| u(\cdot,t) - \overline{u_0} \|_{L^2(\Omega)}
\le
 e^{-\lambda_1 t}\|u_0\|_{L^2(A)},
\qquad t > 0,
\]
where $\overline{u_0}$ denotes the mean value of $u_0$,
\[
\overline{u_0} := \frac{1}{|A|} \int_A u_0(x)\,dx.
\]

Another property of the system \eqref{localpar}-\eqref{nleli} is that it can be obtained as a limit of a parabolic system. 
In fact, we can look at the elliptic equation as the limit of a family of very fast parabolic equations. 
This idea is classical: we replace the stationary equation by a parabolic equation in which the time derivative is multiplied by a small positive parameter $\varepsilon>0$. The smaller $\varepsilon$ is, the faster the dynamics of $v^\varepsilon$ as it relaxes toward its equilibrium. More precisely, for fixed initial data $u_0\in L^2(A)$ and $v_0\in L^2(B)$, we introduce the $\varepsilon$-family $(u^\varepsilon,v^\varepsilon)$, solutions to the following $\varepsilon$--problem. Take, as before, $\Omega=A\cup B$, in $A$ we consider
\begin{equation}\label{elocalpar.intro}
\begin{cases}
u_t^\varepsilon(x,t)
   = \Delta u^\varepsilon(x,t)
     + \displaystyle\int_B 
         J(x-y)\big( v^\varepsilon(y,t) - u^\varepsilon(x,t) \big)\,dy,
   & x\in A,\ t\in(0,T],\\[1.1em]
\displaystyle \frac{\partial u^\varepsilon}{\partial \eta}(x,t)=0,
   & x\in\partial A,\ t\in[0,T],\\[0.7em]
u^\varepsilon(x,0)=u_0(x),
   & x\in A,
\end{cases}
\end{equation}
and in $B$,
\begin{equation}\label{enleli.intro}
\begin{cases}
\varepsilon\, v_t^\varepsilon(x,t)
   = \displaystyle\int_B 
        G(x-y)\big( v^\varepsilon(y,t) - v^\varepsilon(x,t) \big)\,dy
     + \displaystyle\int_A 
        J(x-y)\big( u^\varepsilon(y,t) - v^\varepsilon(x,t) \big)\,dy,
   & x\in B,\ t\in(0,T],\\[1.1em]
v^\varepsilon(x,0)=v_0(x),
   & x\in B.
\end{cases}
\end{equation}
Concerning the limit as $\varepsilon \to 0$ we have the following result.

\begin{theorem}\label{approximation.intro} If the initial conditions $u_0\in L^2(A)$ and $v_0\in L^2(B)$ are fixed. Then, for $T>0$, $u^\epsilon$ converges weakly to $u$ in $L^2 ((0,T),H^1(A))$ and $v^\epsilon$ converges weakly to $v (\cdot, t)$ in $L^\infty((0,T),L^2(B))$, 
with the limits $u$ and $v$ solving the system \eqref{localpar}-\eqref{nleli} with initial datum $u_0\in L^2(A)$.
\end{theorem}

Notice that the initial datum for the fast component, $v_0$, is lost in this limit.

\subsection{Second local-nonlocal model} Now, we introduce the second problem that we study here. This one couples a nonlocal parabolic equation in $B$ with a local elliptic equation in $A$. We still assume that a diffusion process takes place throughout the whole domain $\Omega$, but in this setting the fast diffusion occurs in the 
nonlocal region $B$, while the local region $A$ exhibits a faster dynamic. Under these assumptions, we introduce the second model that now couples a local elliptic equation (driven by the classical Neumann Laplacian) in the region $A$ with a nonlocal parabolic equation (governed by a nonlocal diffusion operator) in the region $B$. The interaction between both regions is, as before, mediated by a nonlocal coupling mechanism.

In the region $A$, we consider the homogeneous Neumann local elliptic problem
\begin{equation}\label{elilocal}
\begin{cases}
\displaystyle 
0 = \Delta u(x,t)
  + \int_B J(x-y)\,\big(v(y,t)-u(x,t)\big)\,dy,
& (x,t)\in A \times (0,\infty),
\\[10pt]
\displaystyle \frac{\partial u}{\partial \eta}(x,t)=0,
& (x,t)\in A \times (0,\infty),
\end{cases}
\end{equation}
and in the region $B$, we consider the nonlocal parabolic problem
\begin{equation}\label{parnl}
\begin{cases}
\displaystyle 
v_t(x,t)
= \int_B G(x-y)\,\big(v(y,t)-v(x,t)\big)\,dy
  + \int_A J(x-y)\,\big(u(y,t)-v(x,t)\big)\,dy,
& (x,t)\in B \times (0,\infty),
\\[10pt]
v(x,0)=v_0(x), & x\in B.
\end{cases}
\end{equation}
These equations are naturally associated with the energy functional $F_v : L^2(A) \to \mathbb{R}$
\[\begin{array}{l}
\displaystyle F_u(v)
= \frac14\iint_{B\times B} G(x-y)\,(v(y)-v(x))^2\,dx\,dy
  + \frac12\iint_{A\times B}  J(x-y)\,v(x)^2\,dy\,dx
  \\[10pt]
  \qquad \qquad \qquad \qquad \qquad \qquad \qquad \qquad \qquad \qquad \qquad \displaystyle - \iint_{A\times B} J(x-y)\,u(x)v(y)\,dy\,dx,
  \end{array}\]
where \(u \in L^2(B)\) is the minimizer of the functional \(E : H^1(A)\to \mathbb{R}\),
\[E(u)= \frac12 \int_A |\nabla u (x)|^2dx + \frac12\int_A\int_B J(x-y)\,(v(y)-u(x))^2\,dy\,dx.\]

For this model we have the following results, which are the analogous to the previous results for the first model.

\begin{theorem} Given $v_0\in L^2(B)$ there exists an unique weak solution $(u,v)$ with $(u(\cdot, t), v(\cdot, t) ) \in H^1(A)\times L^2(B)$ to the problem \eqref{elilocal}--\eqref{parnl}, that is globally defined. The solution $(u,v)$ satisfies the comparison principle: if $v_0\ge \widetilde{v}_0$ then the corresponding solutions are also ordered. 
In addition, solutions preserve the total mass in $B$ along the evolution.
\end{theorem}

Moreover, in the same line of our first model, once the existence and uniqueness of global solutions to the elliptic--parabolic system are established, we look for their asymptotic behavior as $t \to \infty$. We first observe that constant states are stationary solutions of the coupled system. Our elliptic--parabolic system inherits this property. Although one component evolves parabolically while the other satisfies an elliptic nonlocal constraint, the coupling preserves mass in $B$ and we can show
an exponential decay.

\begin{theorem} \label{teo.asymp.2}
Let $v_0 \in L^2(B)$. Let $(u,v)$ be the solution to the elliptic--parabolic coupled system with initial datum $v_0$. Then $v(\cdot,t)$ converges exponentially to its mean value as $t \to \infty$. More precisely, there exists a constant $\lambda_1 > 0$ such that
\[
\| v(\cdot,t) - \overline{v_0} \|_{L^2(\Omega)}
\le
 e^{-\lambda_1 t}\|v_0\|_{L^2(B)},
\qquad t > 0,
\]
where
\[
\overline{v_0} := \frac{1}{|B|} \int_B v_0(x)\,dx
\]
denotes the spatial mean of $v_0$.
\end{theorem}

The system \eqref{elilocal}--\eqref{parnl} can be obtained as a limit of a parabolic system in the following way: we consider a fully parabolic approximation of the coupled local/nonlocal system and prove that, considering a parameter $\varepsilon \to 0$, the solutions converge to a parabolic-elliptic \eqref{elilocal}--\eqref{parnl} limit problem. In this case, we consider $u_0\in L^2(A)$ and $v_0\in L^2(B)$, both fixed, and the following $\varepsilon$-equations, for $\epsilon>0$ small,
\begin{equation}
\begin{cases}
\varepsilon\,  u^\varepsilon_t(x,t)
= \Delta u^\varepsilon(x,t)
+ \displaystyle\int_B J(x-y)\bigl(v^\varepsilon(y,t)-u^\varepsilon(x,t)\bigr)\,dy,
& x\in A,\ t>0,\\[1.2ex]
\displaystyle \frac{\partial u^\varepsilon}{\partial \eta}(x,t)=0,
& x\in \partial A,\ t>0,\\[1.2ex]
u^\varepsilon(x,0)=u_0(x),
& x\in A.
\end{cases}
\end{equation}
\begin{equation}
\begin{cases}
v^\varepsilon_t(x,t)
= \displaystyle\int_B G(x-y)\bigl(v^\varepsilon(y,t)-v^\varepsilon(x,t)\bigr)\,dy \\[0.6em]
\hspace{2.8em}
\displaystyle \qquad \qquad \qquad \qquad + \int_A J(x-y)\bigl(u^\varepsilon(y,t)-v^\varepsilon(x,t)\bigr)\,dy,
& x\in B,\ t>0, \\[1.2ex]
v^\varepsilon(x,0)=v_0(x),
& x\in B.
\end{cases}
\end{equation}

We have the following result,

\begin{theorem}\label{approximation}
Let $u_0\in L^2(A)$ and $v_0\in L^2(B)$ be fixed and also fix $T>0$. Then, $u^\epsilon$ converges weakly to $u$ in $L^2 ((0,T),H^1(A))$ and $v^\epsilon$ converges weakly to $v (\cdot, t)$ in $L^\infty((0,T),L^2(B))$, 
with the limits $u$ and $v$ solving the system \eqref{elilocal}--\eqref{parnl}, with initial datum $v(x,0)=v_0(x)$.
\end{theorem}

Notice that now it is the initial datum for the local component, $u_0$, the one that is lost in the limit.

\bigskip

{\bf Organization of the paper.} First, in Section \ref{sect-1-model} we deal with the parabolic-elliptic case, and next, in Section \ref{sect-2-model} we just
sketch the proofs for the elliptic-parabolic system.

\section{The parabolic-elliptic model} \label{sect-1-model}

In this section we analyze the local--nonlocal, parabolic--elliptic model in which the parabolic part 
is given by the local operator and involves the heat equation. 

\subsection{Energy functional} \label{subsection.energy}

Recall that we introduced the functional $E_v : H^1(A) \to \mathbb{R}$ defined by
\[
E_v(u)
= \frac12\int_A |\nabla u(x)|^2\,dx
  + \frac12\int_A\int_B J(x-y)\,u^2 (x)\,dy\,dx
  - \int_A\int_B J(x-y)\,u(x)v(y)\,dy\,dx,
\]
where \(v \in L^2(B)\) is the minimizer of the functional \(F : L^2(B)\to \mathbb{R}\),
\begin{align}\label{min-v}
F(v)
= \frac14 \iint_{B\times B} G(x-y)\,(v(y)-v(x))^2\,dx\,dy
  + \frac12\int_A\int_B J(x-y)\,(v(y)-u(x))^2\,dy\,dx,
\end{align}
We adopt the convention that $E_v(u)=\infty$ whenever $u \notin H^1(A)$. It is straightforward to verify that these energy functionals are proper, convex, lower semicontinuous and coercive by Lemma \ref{lemma:coerc_v}. 

\begin{lemma} \label{lemma.energia}
Equation \eqref{localpar} is the \(L^2(A)\)-gradient flow of \(E_v\), while equation \eqref{nleli} is the Euler--Lagrange equation associated with the minimization of \(F\). 
\end{lemma}

\begin{proof} Let $u\in H^1(A)$ be fixed and $v\in L^2(B)$ be a minimizer of $F$. 
For any $\phi\in L^2(B)$ and $t\in\mathbb{R}$ we consider the variation
$v_t := v + t\phi$. Since $v$ is a minimizer, we have
\[
\frac{d}{dt} F(v_t)\bigg|_{t=0} = 0.
\]

We compute the variation to obtain:
\begin{align*}
&\frac{d}{dt}
\frac14 \iint_{B\times B} G(x-y)\,(v_t(y)-v_t(x))^2\,dx\,dy
\bigg|_{t=0}+\frac{d}{dt}
\frac12 \iint_{A\times B} J(x-y)\,(v_t(y)-u(x))^2\,dy\,dx
\bigg|_{t=0} \\
&= \frac12 \iint_{B\times B} G(x-y)\,(v(y)-v(x))(\phi(y)-\phi(x))\,dx\,dy+\iint_{A\times B} J(x-y)\,(v(y)-u(x))\,\phi(y)\,dy\,dx.
\end{align*}
Therefore, to find the Euler Equation, i.e.
\[
0 
= \frac12 \iint_{B\times B} G(x-y)\,(v(y)-v(x))(\phi(y)-\phi(x))\,dx\,dy
  + \iint_{A\times B} J(x-y)\,(v(y)-u(x))\,\phi(y)\,dy\,dx,
\]
we must use the symmetry of $G$. The first term can be rewritten as
$$\frac12 \iint_{B\times B} G(x-y)\,(v(y)-v(x))(\phi(y)-\phi(x))\,dx\,dy=-\iint_{B\times B} G(x-y)\,(v(y)-v(x))\phi(x)\,dx\,dy . $$
Hence, we conclude
\[
\int_B \left[
\int_B G(x-y)\,(v(y)-v(x))\,dy
+ \int_A J(x-y)\,(v(x)-u(y))\,dy
\right] \phi(x)\,dx = 0
\]
for every $\phi\in L^2(B)$. Then, the expression inside the brackets must vanish almost everywhere in $B$, that is,
\[
0 = \int_B G(x-y)\,(v(y)-v(x))\,dy
    + \int_A J(x-y)\,(u(y)-v(x))\,dy, 
    \qquad x\in B,
\]
which is precisely the nonlocal elliptic equation \eqref{nleli}, where $v$ is a minimizer of $F$.

Now, we compute the first variation of $E_v$ with respect to $u$.
Let $u\in H^1(A)$ and $\varphi\in H^1(A)$ be arbitrary, and define $u_t := u + t\varphi$. Then
\[
\begin{array}{l}
\displaystyle
\partial E_v(u)=\frac{d}{d t} E_v(u_t)\bigg|_{t=0}
= \int_A \nabla u(x)\cdot \nabla \varphi(x)\,dx
  + \int_A\int_B J(x-y)\,u(x)\,\varphi(x)\,dy\,dx
 \\[10pt]
 \qquad \qquad \qquad \qquad \qquad \qquad \displaystyle - \int_A\int_B J(x-y)\,\varphi(x)\,v(y)\,dy\,dx.
 \end{array}
\]
Integrating by parts the first term (and using the homogeneous Neumann boundary condition), we obtain
\[
\int_A \nabla u\cdot \nabla \varphi\,dx
= - \int_A \Delta u(x)\,\varphi(x)\,dx.
\]
Thus,
\[
\frac{d}{dt} E_v(u_t)\bigg|_{t=0}
= \int_A \Big(
- \Delta u(x)
+ \int_B J(x-y)\,\big(u(x)-v(y)\big)\,dy
\Big)\,\varphi(x)\,dx.
\]
This shows that
\[
\partial E_v(u)(x)
= - \Delta u(x)
  + \int_B J(x-y)\,\big(u(x)-v(y)\big)\,dy.
\]
Hence, the parabolic equation
\[
u_t(x,t)
= \Delta u(x,t)
  + \int_B J(x-y)\,\big(v(y,t)-u(x,t)\big)\,dy
\]
can be written as
\[
u_t(t) + \partial E_v(u(t)) = 0 \quad \text{in } L^2(A),
\]
which is the $L^2(A)$--gradient flow associated with the energy $E_v$. 
\end{proof}

\subsection{Existence and uniqueness}

Now let us establish the existence and uniqueness of solutions for system \eqref{localpar}-\eqref{nleli}. We outline the method used to obtain uniqueness:
First, let $S(\cdot)$ be the continuous semigroup $L^2(A)$ associated with the Laplacian with Neumann boundary condition. The semigroup is a contraction in $L^2(A)$. Then, we define a mild solution  $u\in C([0,\infty);L^2(A))$ of \eqref{localpar}--\eqref{nleli} as
\begin{equation}\label{FVC}
u(t) = S(t)u_0 + \int_0^t S(t-\tau)\,L_B^A(u(\tau))\,d\tau, \qquad t>0,
\end{equation}
where $L_B^A:L^2(A)\to L^2(A)$ is given by the transmission term
\[ L_B^A(u(t))(x) = \int_B J(x-y)\,(v(y,t)-u(x,t))\,dy,\qquad x\in A, \]
where $v\in L^2(B)$ is the function associated with $u$ through the nonlocal elliptic \eqref{nleli}. To prove uniqueness, we use the formulation \eqref{FVC} to obtain the unique mild solution $u\in C([0,t];L^2(A))$. For a fixed $t>0$, given a function $u(\cdot,t)$ defined in $L^2(A)$, we solve the associated elliptic equation \eqref{nleli} in $B$, obtaining a unique weak solution $v(\cdot,t)\in L^2(B)$; this defines the first operator $T_{AB}:C([0,\infty);L^2(A))\to C([0,\infty);L^2(B))$
\[
T_{AB} : u(\cdot,t) \longmapsto v(\cdot,t).
\]
Similarly, given a function $v(\cdot,t)$ defined in $B$ and an initial datum $u_0\in L^2(A)$, we may insert it into the
parabolic problem in $A$, obtaining the corresponding solution $\tilde{u}(\cdot,t)$; this defines
the second operator $T_{BA}:C([0,\infty);L^2(B))\to C([0,\infty);L^2(A))$ defined by
\[
T_{BA} : v(\cdot,t) \longmapsto \tilde{u} (\cdot,t).
\]
This two-step procedure
\[
u(\cdot,t) \;\xrightarrow{T_{AB}}\; v(\cdot,t) \;\xrightarrow{T_{BA}}\; \tilde{u}(\cdot,t)
\]
can be regarded as the composition of the previous two operators, that is, we consider $T : C([0,T];L^2(A))\longrightarrow C([0,T];L^2(A))$ given by
$$T := T_{BA}\circ T_{AB},$$
which assigns to each $u$ the function obtained after solving first the
elliptic problem in $B$ and then the parabolic problem in $A$. Our main result is to show that $T : C([0,T];L^2(A))\longrightarrow C([0,T];L^2(A))$ is a strict contraction, for $t>0$ small enough. It follows from the Banach fixed--point theorem that there exists a unique $u_{fp}$ such that
$$u_{fp} = T(u_{fp}),$$
whose associated function $v=T_{AB}(u_{fp})$ solves the elliptic equation in $B$.
Therefore,
$$(u,v)\in C([0,T];L^2(A))\times C([0,T];L^2(B))$$
is the unique fixed point solution of the full parabolic--elliptic system.

To use this fixed point argument, let us start with the analysis of the first operator, $T_{AB}$. 
Let us fix a time $t>0$ and a function $u(\cdot,t)\in L^{2}(A)$, the following results show that
there exists a unique function $v(\cdot,t)\in L^{2}(B)$ solving the nonlocal elliptic problem
\begin{equation}\label{eq:v-elliptic}
    0=\int_{B} G(x-y)\,\bigl(v(y,t)-v(x,t)\bigr)\,dy
      +\int_{A} J(x-y)\,\bigl(u(y,t)-v(x,t)\bigr)\,dy,
      \qquad x\in B.
\end{equation}

First we state a coercivity propriety of the bilinear form associated to \eqref{eq:v-elliptic}
which guarantees the well-posedness of \eqref{eq:v-elliptic} and plays a fundamental role in
showing existence and uniqueness of solutions.

\begin{lemma}\label{FCoerp-e}\label{lemma:coerc_v}
For every $v \in L^2(B)$, there exists a positive constant $C_c>0$ (independent of $t$) such that
\[
\frac{1}{2}\int_{B}\int_{B} G(x-y)\bigl(v(x)-v(y)\bigr)^2 \,dx\,dy
+ \int_B a(x) (v(x))^2 \,dx \ge C_c \|v\|_{L^2(B)}^2,
\]
where $a$ denotes the function
$$a(x)=\int_A J(x-y)dy,$$
that is nonnegative by Hypothesis \ref{hypo}.
\end{lemma}

\begin{proof}
Suppose, by contradiction, that the coercivity inequality does not hold. 
Then, for each $n \in \mathbb{N}$, there exists $v_n \in L^2(B)$ with $\|v_n\|_{L^2(B)} = 1$ such that
\[
\frac{1}{2}\int_{B}\int_{B} G(x-y)\bigl(v_n(x)-v_n(y)\bigr)^2 \,dx\,dy
+ \int_B a(x) (v_n(x))^2 \,dx \le \frac{1}{n}.
\]
Since both integrals are nonnegative, we have, as $n \to \infty$, 
\[
\frac{1}{2}\int_{B}\int_{B} G(x-y)\bigl(v_n(x)-v_n(y)\bigr)^2 \,dx\,dy \longrightarrow 0,
\quad \text{and} \quad
\int_B a(x) (v_n(x))^2 \,dx \longrightarrow 0.
\]

Since $(v_n)$ is bounded in $L^2(B)$, by the Banach--Alaoglu theorem there exists $v \in L^2(B)$ and a subsequence, still denoted by $(v_n)$, such that
\[
v_n \rightharpoonup v \quad \text{weakly in } L^2(B).
\]
By weak lower semicontinuity and the nonnegativity of $G$, we have
\[
0 = \liminf_{n \to \infty} 
\frac{1}{2}\int_{B}\int_{B} G(x-y)\bigl(v_n(x)-v_n(y)\bigr)^2 \,dx\,dy
\ge \frac{1}{2}\int_{B}\int_{B} G(x-y)\bigl(v(x)-v(y)\bigr)^2 \,dx\,dy.
\]
Hence,
\[
\int_{B}\int_{B} G(x-y)\bigl(v(x)-v(y)\bigr)^2 \,dx\,dy = 0.
\]
Since $G$ is positive and symmetric and $B$ is connected, it follows that $v(x) = \text{const}$ in $B$. 

We also have
\[
0 = \liminf_{n \to \infty} \int_B a(x) (v_n(x))^2 \,dx 
\geq \int_B a(x) (v(x))^2 \,dx,
\]
which implies $\int_B a(x) (v(x))^2 \,dx = 0$.
Since $a(x) \ge 0$, we obtain $v(x) = 0$ almost everywhere in $B$.
In addition, by the positivity of $a(x)$ for points in $B$ close to $A$ (recall that we assumed that
$dist(A,B) < supp(J)$) we have strong convergence in $L^2 (B_1)$ for $B_1 = \{x\in B : dist(x,A) < supp(J)\}$.
Now, let $B_2 = \{x\in B : dist(x,B_1) < supp(G)\}$. we have
\[
\begin{array}{l}
\displaystyle 0 = \liminf_{n \to \infty} 
\frac{1}{2}\int_{B}\int_{B} G(x-y)\bigl(v_n(x)-v_n(y)\bigr)^2 \,dx\,dy \\[10pt]
\displaystyle
\geq \liminf_{n \to \infty} 
\frac{1}{2}\int_{B_1}\int_{B_2} G(x-y)\bigl(v_n(x)-v_n(y)\bigr)^2 \,dx\,dy
\\[10pt]
\displaystyle = 
\liminf_{n \to \infty} 
\frac{1}{2}\int_{B_1}\int_{B_2} G(x-y) (v_n)^2 (x) \,dx\,dy
+ \frac{1}{2}\int_{B_1}\int_{B_2} G(x-y) (v_n)^2 (y) \,dx\,dy 
\\[10pt]
\displaystyle \qquad \qquad - 
\int_{B_1}\int_{B_2} G(x-y) (v_n(x))(v_n(y)) \,dx\,dy.
\end{array}
\]
Hence, since we have strong convergence in $L^2 (B_1)$ and weak
convergence in $L^2 (B_2)$ to zero, we get
$$
\liminf_{n \to \infty} 
\frac{1}{2}\int_{B_2} \left(\int_{B_1} G(x-y) \, dy \right) (v_n)^2 (x) \,dx =0,
$$
from where we obtain that $v_n$ converges strongly in $L^2 (B_2)$. 
Iterating this idea a finite number of times (we assumed that $B$ is connected and bounded), we conclude that
$v_n$ converges strongly to zero in $L^2 (B)$, but this contradicts the fact that $\|v_n\|_{L^2(B)} = 1$.
Therefore, the coercivity inequality must hold.
\end{proof}

Using Lax-Milgram's Theorem, we can prove the uniqueness of \eqref{eq:v-elliptic}.

\begin{proposition}\label{v(u)existence}
Let $t>0$ be fixed. For each \(u(\cdot,t)\in L^2(A)\), the equation
\begin{equation}\label{eq:elliptic_v}
\int_B G(x-y)\bigl(v(y,t)-v(x,t)\bigr)\,dy
\;+\; \int_A J(x-y)\bigl(u(y,t)-v(x,t)\bigr)\,dy \;=\; 0
\end{equation}
has a unique solution \(v(\cdot,t)\in L^2(B)\).  Hence, for each $u\in C([0,T];L^2(A))$, we define
\[ T_{AB}\bigl(u)(t):=T_{AB}(u(\cdot,t))=v(\cdot,t), \quad \text{for all}\quad  t\in [0,T]
\]
and we have that $T_{AB}$ is a well-defined linear bounded operator.
\end{proposition}

\begin{proof}For $v,\varphi\in L^2(B)$ we consider the bilinear form
\[
\mathcal{B}\bigl(v,\varphi\bigr)
:=
\frac{1}{2}\int_{B}\int_{B} G(x-y)\bigl(v(x)-v(y)\bigr)\bigl(\varphi(x)-\varphi(y)\bigr)\,dx\,dy
+ \int_B a(x)\,v(x)\,\varphi(x)\,dx,
\]
where
\[
a(x):=\int_A J(x-y)\,dy \ge 0 .
\]
For the right-hand side we define, for each fixed $t>0$, the linear functional (depending on \(u(\cdot,t)\))
\[
\ell_{u}\bigl(\varphi\bigr)
:= \int_B\int_A J(x-y)\,u(y)\,\varphi(x)\,dy\,dx.
\]
Since \(G\) is nonnegative, we have
\[
\mathcal{B}\bigl(v,v\bigr)
= \frac{1}{2}\int_{B}\int_{B} G(x-y)\bigl(v(x)-v(y)\bigr)^2 dxdy
+ \int_B a(x)\, v(x)^2 dx \;\ge\; 0.
\]
Under the assumptions in Lemma~\ref{lemma:coerc_v}, there exists \(C_c>0\) such that
\[
\mathcal{B}\bigl(v,v\bigr) \;\ge\; C_c\,\|v\|_{L^2(B)}^2,
\qquad \forall\,v\in L^2(B).
\]
Moreover, using that $G,J$ are continuous, one can easily check that
\(\mathcal{B}\) is continuous in \(L^2(B)\times L^2(B)\) and that
\(\ell_{u}\) is a continuous linear functional in \(L^2(B)\). 
Hence all hypotheses of Lax--Milgram's theorem (in the Hilbert space \(L^2(B)\))
are satisfied and we conclude that there exists a unique \(v\in L^2(B)\) such that
\[
\mathcal{B}\bigl(v,\varphi\bigr)
= \ell_{u}\bigl(\varphi(\cdot)\bigr),
\qquad \forall\,\varphi\in L^2(B),
\]
which is exactly the weak formulation of \eqref{eq:elliptic_v}. By Lax--Milgram we
also obtain the estimate
\[
\|v\|_{L^2(B)}
\;\le\; \frac{1}{C_c}\,\|\ell_{u}\|_{(L^2(B))'}
\;\le\; \frac{\|J\|_{L^2(A\times B)}}{C_c}\,\|u\|_{L^2(A)}.
\]
Therefore, the operator \(T_{AB}\) is linear and bounded from
\(L^2(A)\) to \(L^2(B)\). Restricting the domain, we may also view
\(T_{AB}:H^1(A)\to L^2(B)\) as a bounded linear operator.
\end{proof}

For $T>0$ we will use the Banach spaces
\[
C\big([0,T];L^{2}(A)\big), \qquad 
C\big([0,T];L^{2}(B)\big),
\]
endowed with the norms
\[
\|u\|_{C([0,T];L^{2}(A))}
:= \sup_{s\in[0,T]} \|u(\cdot,s)\|_{L^{2}(A)},\qquad
\|v\|_{C([0,T];L^{2}(B))}
:= \sup_{s\in[0,T]} \|v(\cdot,s)\|_{L^{2}(B)}.
\]
It follows from Proposition~\ref{v(u)existence} that the operator $T_{AB}:L^{2}(A)\to L^{2}(B)$ 
is linear and bounded. The following result guarantees that $T_{AB}$ is a Lipschitz operator.

\begin{proposition}[Existence and Uniqueness for the Nonlocal Elliptic Equation]\label{EUnlp-e}
Let $T>0$ and let $u\in C([0,t];L^{2}(A))$. Take
\[v=T_{AB}(u)\in C([0,T];L^{2}(B))
\]
which solves the nonlocal elliptic problem \eqref{nleli} for every $s\in[0,T]$. 

Given $u_{1},u_{2}\in C([0,T];L^{2}(A))$ and letting 
$v_{1}=T_{AB}(u_1)$ and $v_{2}=T_{AB}(u_2)$, there exists a constant 
$C_e:=C_e(J,G)>0$ such that
\begin{equation}\label{vunormp-e}
    \|v_{1}-v_{2}\|_{C([0,T];L^{2}(B))}
    \;\le\; C_e\,\|u_{1}-u_{2}\|_{C([0,T];L^{2}(A))}.
\end{equation}
In particular, the solution operator $u\mapsto T_{AB}(u)$ is Lipschitz continuous from 
$L^{2}(A)$ to $L^{2}(B)$ uniformly in time.
\end{proposition}

\begin{proof}
Fix $s\in[0,T]$. By Proposition~\ref{v(u)existence}, for each
$u(\cdot,s)\in L^{2}(A)$ there exists a unique solution
\[
v(\cdot,s)=T_{AB}(u(\cdot,s))\in L^{2}(B)
\]
of the nonlocal elliptic equation \eqref{nleli}. To prove the Lipschitz
estimate, let $u_{1},u_{2}\in C([0,T];L^{2}(A))$ and define
\[
v_{1}(\cdot,s)=T_{AB}(u_{1}(\cdot,s)), \qquad
v_{2}(\cdot,s)=T_{AB}(u_{2}(\cdot,s)).
\]
Set $w(\cdot,s)=v_{1}(\cdot,s)-v_{2}(\cdot,s)$. Subtracting the two
equations satisfied by $v_{1}$ and $v_{2}$ gives
\[
\int_{B}G(x-y)\bigl(w(y,s)-w(x,s)\bigr)\,dy
+
\int_{A}J(x-y)\bigl(u_{1}(y,s)-u_{2}(y,s)-w(x,s)\bigr)\,dy
=0,
\qquad x\in B.
\]

We now multiply the above identity by $w(x,s)$ and integrate over $B$. Using
the symmetric structure of the kernel $G$, we can apply Green Identity,
$$
	\int_{B}\int_{B} G(x-y)\,(w(y)-w(x))\,w(x)\,dy\,dx
	= -\frac{1}{2}\int_{B}\int_{B} G(x-y)\,(w(y)-w(x))^2\,dy\,dx.
$$
 for the nonlocal term and
writing $a(x)=\int_{A}J(x-y)\,dy$, we obtain
$$
\begin{array}{l}
\displaystyle \frac12\!\int_{B}\!\int_{B}G(x-y)\bigl(w(y,s)-w(x,s)\bigr)^{2}\,dx\,dy
\;+\;\int_{B}a(x)\,w(x,s)^{2}\,dx
\\[10pt]
\qquad \displaystyle =
\int_{B}\int_{A}J(x-y)\bigl(u_{1}-u_{2}\bigr)(y,s)\,w(x,s)\,dy\,dx.
\end{array}
$$
By Lemma~\ref{lemma:coerc_v}, there exists $C_{c}>0$
such that
\[
C_{c}\,\|w(\cdot,s)\|_{L^{2}(B)}^{2}
\;\le\;
\frac12\!\int_{B}\!\int_{B}G(x-y)\bigl(w(y,s)-w(x,s)\bigr)^{2}\,dx\,dy
+\int_{B}a(x)\,w(x,s)^{2}\,dx.
\]
Combining with the previous identity we obtain
\[
C_{c}\|w(\cdot,s)\|_{L^{2}(B)}^{2}
\;\le\;
\int_{B}\int_{A}J(x-y)\bigl(u_{1}-u_{2}\bigr)(y,s)\,w(x,s)\,dy\,dx.
\]
Applying Cauchy--Schwarz's inequality and Fubini's theorem, we get
\begin{align*}
C_{c}\|w(\cdot,s)\|_{L^{2}(B)}^{2}
&\le
\|w(\cdot,s)\|_{L^{2}(B)}
\bigg(
\int_{B}
\Big|\int_{A}J(x-y)(u_{1}-u_{2})(y,s)\,dy\Big|^{2}
dx
\bigg)^{1/2} \\
&\le |B|
\|w(\cdot,s)\|_{L^{2}(B)}\,
\Big(
\int_{B}\int_{A}J(x-y)^{2}\,dy\,dx
\Big)^{1/2}
\|u_{1}(\cdot,s)-u_{2}(\cdot,s)\|_{L^{2}(A)}.
\end{align*}

If $w(\cdot,s)\equiv 0$, the desired inequality holds trivially. Otherwise we
divide both sides by $\|w(\cdot,s)\|_{L^{2}(B)}$ and obtain
\[
\|v_{1}(\cdot,s)-v_{2}(\cdot,s)\|_{L^{2}(B)}
\le
C_e\,\|u_{1}(\cdot,s)-u_{2}(\cdot,s)\|_{L^{2}(A)},
\]
where
\[
C_e=\frac{1}{C_{c}}
\bigg(\int_{B}\int_{A}J(x-y)^{2}\,dy\,dx\bigg)^{1/2}>0.\]
Taking the supremum in $s\in[0,T]$ yields the uniform-in-time estimate
\[
\|v_{1}-v_{2}\|_{C([0,T];L^{2}(B))}
\;\le\;
C_e\,\|u_{1}-u_{2}\|_{C([0,T];L^{2}(A))}.
\]

Finally, to prove that $v\in C([0,T];L^{2}(B))$, let $s_{n}\to s$ in $[0,T]$.
Then
\[
\|v(\cdot,s_{n})-v(\cdot,s)\|_{L^{2}(B)}
=
\|T_{AB}(u(\cdot,s_{n})-u(\cdot,s))\|_{L^{2}(B)}
\le
C_e\,\|u(\cdot,s_{n})-u(\cdot,s)\|_{L^{2}(A)}\to 0,
\ \text{as} \ n\to \infty,\]
since $u\in C([0,T];L^{2}(A))$. Thus $v\in C([0,T];L^{2}(B))$ and the proof is complete.
\end{proof}

\begin{remark}\label{vregularity} {\rm
We can describe the regularity of $v$ in terms of the kernels $J$ and $G$ as follows: let us write the nonlocal elliptic equation \eqref{nleli} explicitly
\[
0 \;=\; \int_{B} G(x-y)\bigl(v(y)-v(x)\bigr)\,dy
      +\int_{A} J(x-y)\bigl(u(y)-v(x)\bigr)\,dy.
\]
Since the kernels are nonnegative and not identically zero by Hypothesis \eqref{hypo}, we may rewrite this as
\begin{equation}\label{eq:v-formula}
v(x,t)
    =\frac{\displaystyle\int_{B} G(x-y)\,v(y)\,dy
           +\displaystyle\int_{A} J(x-y)\,u(y)\,dy}
           {\displaystyle\int_{B} G(x-y)\,dy
           +\displaystyle\int_{A} J(x-y)\,dy}.
\end{equation}

If we further assume that $J$ and $G$ are smooth, then the right-hand side of 
\eqref{eq:v-formula} is a convolution of $u$ and $v$ with smooth kernels.
Since $u\in L^{2}(A)$ and $v\in L^{2}(B)$, it follows that $v$ 
inherits the regularity of the kernels. In particular, $v$ becomes smooth whenever $J$ and $G$ are smooth. }
\end{remark}

Before analyzing the parabolic equation \eqref{localpar}, motivated by the semi-group approach, for a given $T>0$ to be chosen later and for a fix $z\in C([0,T];L^2(B))$, we define the operator,
\[
T_{1}: C([0,T];L^{2}(A)) \longrightarrow C([0,T];L^{2}(A))
\]
by
\begin{align}\label{def-T1}
(T_{1}u)(t)
    &= w(t) \\
    &= S(t)u_{0}
       + \int_{0}^{t} S(t-\tau)
          \left[
             \int_{B} J(\cdot-y)\, z(y,\tau)\,dy
             - \int_{B} J(\cdot-y)\, u(\cdot,\tau)\,dy
          \right] d\tau,
       \qquad t\in[0,T]. \nonumber
\end{align}

%where $w$ is a solution to 
%\begin{equation} w(x,t)=S(t)u_0+\int_{0}^{T} S(t-\tau)\int_B J(x-y)(z(u,\tau)-w(x,\tau))dyd\tau, (x,t)\in A\times (0,T), \end{equation}
This representation formula in terms of the heat semigroup naturally motivates the definition of the operator $T_{BA}$, which helps to prove existence and uniqueness via a fixed-point argument.

\begin{proposition}[Existence and Uniqueness for the Local Parabolic Equation]
\label{EULocalp-e}
Let $T>0$ be fixed and take $v\in C([0,T];L^{2}(B))$.  
Then given $u_0\in L^2(A)$, the parabolic problem \eqref{localpar} admits a unique solution
\[
\tilde{u} \in C([0,T];L^2(A)).
\]
In particular, this defines a linear operator
\[
T_{BA} : C([0,T];L^{2}(B)) \longrightarrow C([0,T];\textcolor{blue}{L^{2}(A)}),
\qquad 
T_{BA}(v)=\tilde{u},
\]
which associates to each $v$ the unique solution $\tilde{u}$ of
\eqref{localpar} on $[0,T]$.

Moreover, if $v_{1},v_{2}\in C([0,T];L^{2}(B))$ and 
$\tilde{u}_{1}=T_{BA}(v_{1})$, $\tilde{u}_{2}=T_{BA}(v_{2})$, then there exists a constant
$C_e(t)>0$ such that
\begin{equation}\label{uvnormp-e}
    \|\tilde{u}_{1}-\tilde{u}_{2}\|_{C([0,T];L^{2}(A))}
    \;\le\;
    C_e(t)\,\|v_{1}-v_{2}\|_{C([0,T];L^{2}(B))}.
\end{equation}
\end{proposition}

\begin{proof} 
Consider $T_{1}:C([0,T]; L^2(A))\to C([0,T]; L^2(A))$ and $z\in C([0,T]; L^2(B))$ given by
$$(T_1u_1)(t)=S(t){u_0^1}+\int_{0}^{t}S(t-\tau)\bigg[\int_BJ(x-y)z(y,\tau)dy-\int_BJ(x-y)u_1(x,\tau)dy)\bigg] d\tau,$$
where $S:L^2(A)\to L^2(A)$ is the heat semigroup associated with the Laplacian with Neumann boundary conditions 
(that has norm $\|S\|=1$). Since $\int_{\mathbb{R}^N}J(z)dz=1$, taking $u_1,u_2 \in C([0,T];L^2(A))$ we have
\begin{equation}
\begin{array}{l}
\displaystyle \|(T_1u_1 -T_1 u_2) (t)\|_{L^2(A)}
\displaystyle \leq \bigg\|S(t){(u^1_{0}-u_0^2)}-\int_{0}^{t}S(t-\tau)\bigg[\int_BJ(x-y)(u_2(x,\tau)-u_1(x,\tau)dy)\bigg] d\tau\bigg \|_{L^2(A)}\\[10pt]
\displaystyle 
\qquad \qquad \leq {\|(u^1_{0}-u_0^2)\|_{L^2(A)}}+\int_{0}^{t}\|u_2(x,\tau)-u_1(x,\tau)\|_{L^2(A)}d\tau \\[10pt]
\displaystyle 
\qquad \qquad 
\leq {\|(u^1_{0}-u_0^2)\|_{L^2(A)}}+t\sup_{\tau \in [0,T]}\|u_2(x,\tau)-u_1(x,\tau)\|_{L^2(A)}.
\end{array}
\end{equation}
Then, we conclude that
\begin{align}
\|T_1u_1-T_1u_2\|_{C([0,T];L^2(A))}\leq \|u^1_{0}-u_0^2\|_{L^2(A)}+T\|u_2-u_1\|_{C([0,T];L^2(A))}.
\end{align}

Now, choose $T<1$ and $u_0^1=u_0^2=u_0$ in the above inequality, to obtain that $T_1$ is a strict contraction in $C([0,T];L^2(A))$ for each fixed $z\in C([0,T];L^2(B))$ and then the existence and uniqueness part of the theorem follows from Banach fixed-point theorem in $[0,T]$. To extend the solution to $[0,\infty)$, we consider $u(x,t)\in L^2(A)$ as initial datum of \eqref{localpar}
and obtain a solution up to $[0,2T]$. Iterating this procedure we get a solution defined for all $t\in (0,\infty)$.

Now, consider $T_{BA}:C([0,T]; L^2(B))\to C([0,T]; \textcolor{blue}{L^2(A)})$ given by
$(T_{BA}v)(t)=w$ which is defined as before by
$$w(t,x)=S(t)u_0+\int_{0}^{t}S(t-\tau)\bigg[\int_BJ(x-y)v(y,\tau)dy-\int_BJ(x-y)w(x,\tau)dy)\bigg] d\tau.$$
This is well defined by Lemma \eqref{EULocalp-e}.
Then considering two initial data $w_0^1,w_0^2\in L^2(A)$, we get
$$
\begin{array}{rl}
\displaystyle \|(T_{BA}v_1(\cdot,t)-T_{BA}v_2(\cdot,t)\|_{L^2(A)}
& \displaystyle \le 
\|S(t)(w_0^1 - w_0^2)\|_{L^2(A)}  \\[10pt]
& \displaystyle \qquad + 
\left\|
\int_0^t 
S(t-\tau)\!\left[
   \int_B J(x-y)\big( v_1(y,\tau) - v_2(y,\tau) \big)\,dy
\right] d\tau
\right\|_{L^2(A)} \\[10pt]
& \displaystyle  \qquad +
\left\|
\int_0^t 
S(t-\tau)\!\left[
   \int_B J(x-y)\,dy \, \big( w_1(x,\tau) - w_2(x,\tau) \big)
\right] d\tau
\right\|_{L^2(A)} .
\end{array}
$$

Using the contractivity of the semigroup, we obtain
\begin{align*}
\|T_{BA}v_1(\cdot,t)-T_{BA}v_2(\cdot,t)\|_{L^2(A)} 
&\le 
\|w_0^1 - w_0^2\|_{L^2(A)}   \\
&\quad +
\int_0^t
\left\|
  \int_B J(x-y)\big( v_1(y,\tau) - v_2(y,\tau) \big)\,dy
\right\|_{L^2(A)} d\tau \\
&\quad +
\int_0^t
\left\|
  \int_B J(x-y)\,dy \, \big( w_1(x,\tau)-w_2(x,\tau) \big)
\right\|_{L^2(A)} d\tau .
\end{align*}

Hence,
\begin{align*}
\|w_1 - w_2\|_{C([0,T];L^2(A))}
&= 
\|T_{BA}u_1 - T_{BA}u_2\|_{C([0,T];L^2(A))}
\\
&\le \|w_0^1 - w_0^2\|_{L^2(A)} \\
& \qquad 
+ t C_J |A| \|v_1 - v_2\|_{C([0,T];L^2(B))}
+ t C_J |B| \|w_1 - w_2\|_{C([0,T];L^2(A))}.
\end{align*}
Taking $w_1 = w_2 = u_0\in L^2(A)$, 
we obtain
\[
\|w_1 - w_2\|_{C([0,T];L^2(A))}
\le 
\underbrace{
\frac{t\, C_J |A|}{1 - t\, C_J |B|}
}_{=:C(t)}
\,
\|v_1 - v_2\|_{C([0,T];L^2(B))},
\]
which yields inequality \eqref{uvnormp-e}, that is valid for $t\in [0,T]$ such that $0<1 - t\, C_J |B|<1$.
\end{proof}

Finally, we get the existence and uniqueness for the whole 
system \eqref{localpar}-\eqref{nleli}, combining the two previous lemmas.

\begin{theorem}\label{EUp-e}
Let $u_0\in L^2(A)$.
Then, there exists a unique fixed point solution 
\[
u\in C([0,\infty);L^2(A)), \qquad v\in C([0,\infty);L^2(B)),
\]
to the coupled parabolic-elliptic problem
\eqref{localpar}--\eqref{nleli}.
\end{theorem}

\begin{proof} Fix $T>0$.
Let $T_{AB}$ and $T_{BA}$ be the operators introduced in
Lemma~\ref{EUnlp-e} and Lemma~\ref{EULocalp-e}, respectively, and define
the composition
\[
T := T_{BA}\circ T_{AB}.
\]

From Lemma~\ref{EUnlp-e}, we have that for any
$u_1,u_2\in C([0,T];L^2(A))$ and
$v_1=T_{AB}(u_1)$, $v_2=T_{AB}(u_2)$,
there exists a constant $C_e>0$ depending only on the kernels
$J$ and $G$ such that
\begin{equation}\label{eq:AB-lip}
\|v_1-v_2\|_{C([0,T];L^2(B))}
\;\le\; C_e\,\|u_1-u_2\|_{C([0,T];L^2(A))}.
\end{equation}

Similarly, from Lemma~\ref{EULocalp-e}, for any
$v_1,v_2\in C([0,T];L^2(B))$ and
$u_1=T_{BA}(v_1)$, $u_2=T_{BA}(v_2)$,
there exists a constant $C_e(t)>0$, depending on $t$, such that
\begin{equation}\label{eq:BA-lip}
\|u_1-u_2\|_{C([0,T];L^2(A))}
\;\le\; C_e(t)\,\|v_1-v_2\|_{C([0,T];L^2(B))}.
\end{equation}

Let $u_1,u_2\in C([0,T];L^2(A))$ and set
$v_1=T_{AB}(u_1), \ v_2=T_{AB}(u_2)$. Then
\[
w_1 := T(u_1)=T_{BA}(v_1),\qquad 
w_2 := T(u_2)=T_{BA}(v_2).
\]
Applying first \eqref{eq:BA-lip} and then \eqref{eq:AB-lip} gives
$$
\begin{array}{rl}
\displaystyle \|T(u_1)-T(u_2)\|_{C([0,T];L^2(A))}
& \displaystyle = \|w_1-w_2\|_{C([0,T];L^2(A))}  \\[10pt]
& \displaystyle \le C_e(t)\,\|v_1-v_2\|_{C([0,T];L^2(B))}    \\[10pt]
&\displaystyle \le (C_e(t)) C_e\,\|u_1-u_2\|_{C([0,T];L^2(A))}.
\end{array}
$$
Thus $T$ is a Lipschitz operator with constant
$C_e\,(C_e(t))>0$.

Since $C_e$ is fixed and $C_e(t)\to 0$ as $t\downarrow 0$
(from the explicit computation made in Lemma~\eqref{EULocalp-e}),
we may choose $t_0>0$ sufficiently small such that
\[
L(t_0) = C_e\,C_e(t_0) <1.
\]
Hence, for such $t_0$, the operator
\[
T : C([0,t_0];L^2(A)) \to C([0,t_0];L^2(A))
\]
is a contraction and by Banach's fixed point theorem, the map $T$ has a unique fixed point
\[
u\in C([0,t_0];L^2(A)),
\qquad u = T(u).
\]
Setting $v=T_{AB}(u)\in C([0,t_0];L^2(B))$,
the pair $(u,v)$ satisfies the parabolic--elliptic system
\eqref{localpar}--\eqref{nleli} on $[0,t_0]$.
Repeating the above argument on consecutive intervals of length $t_0$,
we obtain a unique fixed point solution on any finite interval $[0,T]$.
In particular, existence and uniqueness hold on $[0,T]$ for the original $T>0$.

Repeating the above argument on consecutive intervals of length t0, we obtain a unique fixed point
solution on any finite interval [0,T], obtaining global existence and uniqueness.
\end{proof}

\subsubsection{Mass conservation}

In agreement with what is understood as Neumann boundary conditions,
we show that the total mass of the solution is preserved in time.

\begin{proposition}\label{massconervationp-e} For every $u_0\in {L^2(A)}$ and $t>0$, the unique solution of \eqref{localpar} preserves the total mass in $A$, that is
$$\int_A u(y,t)dy=\int_A u_0(y)dy.$$
\end{proposition}

\begin{proof}
  Notice that, integrating \eqref{localpar} in space and time we get
   $$
    \int_A u(x,t)\, dx-\int_Au_0(x)\, dx=\int_A\int_{0}^{t}\int_B J(x-y)(v(y,t)-u(x,t))\, dy \, dt \, dx .
$$
From \eqref{nleli}, using the simmetry of $G$, we get 
$$\int_A\int_{0}^{t}\int_B J(x-y)(v(y,t)-u(x,t))\,  dy \, dt \, dx=-\int_B \int_{0}^{t}\int_B G(x-y)(v(y,t)-v(x,t))
\, dy\, dt\, dx = 0.$$
Hence, for all $t>0$ 
$$\int_A u(x,t)\, dx=\int_Au_0(x)\, dx.$$
\end{proof}

\begin{remark} {\rm Proposition \eqref{massconervationp-e} also says that, if we consider $$w(x,t)=u(x,t)-\displaystyle\frac{1}{|A|}\int_A u_0(x)dx$$ we get that $w$ is the solution to our problem \eqref{localpar}-\eqref{nleli} with initial datum $$w_0(x,0)=u_0(x)-\displaystyle \frac{1}{|A|}\int_A u_0(x)dx. $$ Then, when we study the asymptotic behavior
as $t\to \infty$ we can consider, without loss of generality, the problem \eqref{localpar}-\eqref{nleli} with $\int_A u(x,t)dx=0$ for all $t\geq 0$. }
\end{remark}

\subsection{Comparison principle}

The comparison principle also holds for this problem. 
This follows using the same method used previously to obtain existence and uniqueness 
of solutions. For the sake of completeness, we will include some details here.

\begin{proposition}\label{comparison1} Let $u_0\in L^2(A)$ be an initial datum for \eqref{localpar} with $u_0\geq 0$. If $u$ is a solution of \eqref{localpar}, then $u\geq 0$.
\end{proposition}

\begin{proof}
Consider $u_1(x,t)\geq 0$ a function in $C([0,T];L^2(A))$ to start our fixed point argument. Then, by Proposition \ref{EUnlp-e}, we obtain a function $v_1\in C([0,t];L^2(B))$ that satisfies the equation
\begin{equation}
    0=\int_B G(x-y)(v_1(y,t)-v_1(x,t))dy+\int_A J(x-y)(u_1(y,t)-v_1(x,t))dy.
\end{equation}
For each fixed $t>0$ with $v_1(\,\cdot\,,t)\in L^2(B)$ and considering the Hypothesis \ref{hypo} on the kernels $J$ and $G$, the convolutions are continuous functions on $\overline{B}$.  

%By Remark~\eqref{vregularity}, $v_1(\,\cdot\,,t)$ is continuous on $\overline{B}$ and, in particular,
%attains a maximum at some point $x_1\in\overline{B}$, and
%$$\int_B G(x_1-y)\underbrace{(v_1(x_1,t)-v_1(y,t)}_{\geq 0})dy+\underbrace{\int_{A}J(x_1-y)dy}_{\geq 0}v_1(x_1,t)dx=\underbrace{\int_AJ(x_1-y)u_1(y,t)dy}_{\geq 0}.$$
%Then, we conclude that
%$$v_1(x_1,t)\ge 0.$$ 

Now, considering $x_0\in \overline{B}$ as a minimum point of $v_1$, we get that $v_1(x_0,t)$ is nonnegative, since
it holds that 
$$\underbrace{\int_{A}J(x_0-y)dy}_{\geq 0}v_1(x_0,t)=\int_B G(x_0-y)\underbrace{(v_1(y,t)-v_1(x_0,t)}_{\geq 0})dy+\underbrace{\int_AJ(x_0-y)u(y,t)dy}_{\geq 0} \geq 0.$$
We conclude that 
$$v_1\geq 0.$$ 

Applying Proposition \ref{EULocalp-e}, we get that there exists a solution $u_2\in C([0,T];L^2(A))$ that solves
    \begin{equation}
\begin{cases}
    {(u_2)}_t(x,t)=\Delta u_2(x,t) -u_2(x,t)\displaystyle\int_B J(x-y)dy+ \displaystyle\int_B J(x-y)v_1(y,t)dy, \qquad x\in A\times (0,T)  \\
\displaystyle    \frac{\partial u_2}{\partial \eta}(x,t)=0, \qquad x\in \partial A \times (0,T)\\
    u_2(x,0)=u_0(x), \qquad x\in A
\end{cases}\end{equation}
Since $\int_B J(x-y)v_1(y)dy\ge 0$ and the operator $\Delta-b(x)$, $b(x)=\int_B J(x-y)dy$, with Neumann boundary conditions generates a positive semigroup in $L^2(A)$ (see \cite{ouhabaz}, p.~49), by the variation of constants formula, if $u_0\ge 0$, 
then the solution satisfies $$u_2(t)\ge 0$$ for all $t\ge 0$.
Iterating this procedure using the fixed point operator $T:C([0,T];L^2(A))\to C([0,T];L^2(A))$ introduced in the proof of Theorem \ref{EUp-e}, we conclude that the unique fixed point $u$ is nonnegative.
\end{proof}

As a consequence of Proposition \ref{comparison1}, we present an example showing that solutions of \eqref{localpar}--\eqref{nleli} may fail to be continuous across the interface, when $\partial A\,\cap\, \partial B\neq \emptyset.$

\begin{remark} {\rm
By Hypothesis~\ref{hypo} on the kernels $J$ and $G$, any solution 
$v(\cdot,t)\in L^2(B)$ of \eqref{nleli} is continuous for every $t>0$,  
as stated in Remark~\ref{vregularity}.  

To look for the regularity of $u(\cdot,t)$, notice that it solves the linear parabolic problem 
inside a smooth domain $A$,
 \begin{equation*}
\begin{cases}
u_t(x,t)
   = \Delta u(x,t)- b(x)u(x,t) + g(x,t),
   & x\in A,\ t>0,\\[0.4em]
\dfrac{\partial u}{\partial \eta}(x,t)=0,
   & x\in\partial A,\ t>0,\\[0.4em]
u(x,0)=u_0(x),
   & x\in A,
\end{cases}
\end{equation*}
where
\[
g(x,t)
   := \displaystyle\int_B J(x-y)v(y,t)\,dy
 , \qquad b(x)=\int_B J(x-y)dy.
\]
and $u_0\in C(\overline{A})$. For each fixed $t>0$ we have $g(\cdot,t)\in  C(\overline A)$. By \cite[Theorem~4.31]{Lunardi}, the solution $u$ belongs to
$C((0,T];C(\overline A))$ whenever $u_0\in C(\overline A)$ and 
$g\in C([0,T];C(\overline A))$.
\medskip
We are now ready to construct a counterexample showing that solutions of
\eqref{localpar}--\eqref{nleli} may be discontinuous across the interface, when we are dealing with the case
$\partial A\cap \partial B \neq \emptyset$.  
Consider the one-dimensional setting with
\[
A=(-1,0) \qquad \text{and} \qquad B=(0,1).
\]
and $u_0(x)=1$ in $A$. Then, $u\in C((0,T];C(\overline{A}))$ and $v\in C([0,T];C(\overline{B}))$ are mild solutions satisfying \eqref{localpar}-\eqref{nleli}. We want to show that $u(0,t)\neq v(0,t)$ for $t>0$ small. The equation for $u$ is
\[
u_t = u_{xx}+\int_0^1 J(x-y)(v(y,t)-u(x,t))\,dy, \qquad x\in (-1,0).
\]
Since $u_0\equiv1$, we have $u_{xx}(x,0)=0$ and therefore
\begin{equation}\label{u_t(x,0)}
u_t(x,0)=\int_0^1 J(x-y)(v_0(y)-1)\,dy, \qquad x\in (-1,0).
\end{equation}
Since $v(.,t)$ solves 
\begin{align}\label{v_0solves}
0=\int_0^1 G(x-y)(v(y,t)-v(x,t))\,dy
+\int_{-1}^0 J(x-y)(u(x,t)-v(x,t))\,dy, \qquad x\in (0,1).
\end{align}
Then $v_0\in L^2(B)$ satisfies the following expression
\begin{align}\label{v_0}
v_0(x)\bigg[\int_{0}^{1}G(x-y)\,dy+\int_{-1}^{0}J(x-y)\,dy\bigg]=\int_0^1 G(x-y)v_0(y)\,dy
+\int_{-1}^0 J(x-y)\,dy, \qquad x\in (0,1).\end{align}
Since all kernels are continuous, then \eqref{v_0} give us that $v_0$ is continuous in $(0,1)$, then there exists $M=\max_{x\in [0,1]}v_0(x):=v_0(\tilde x)$ and
\begin{align*}
    M\bigg[\int_{0}^{1}G(\tilde x-y)\,dy+\int_{-1}^{0}J(\tilde x-y)\,dy\bigg]&=\int_0^1 G(\tilde x-y)v_0(y)\,dy
+\int_{-1}^0 J(\tilde x-y)\,dy\\
&\le M\int_0^1 G(\tilde x-y)\,dy
+\int_{-1}^0 J(\tilde x-y)\,dy.
\end{align*}
Then we have $M\leq 1$. Consider that $v_0\not\equiv1$ \footnote{otherwise, if $v_0\equiv 1$ we have constant solutions ($u=1$ and $v=1$) for the Neumann problem, a case where the solutions are continuous on the interface $\Gamma$.} , there exists $y_0\in(0,1)$ such that $v_0(y_0)<1$. Since $J\ge0$ and is positive on a set of positive measure, it follows by \eqref{u_t(x,0)} that
\[
u_t(x,0)<0 .
\]
In particular $u_t(0,0)<0$. By continuity of $u$, there exists $t_1>0$ such that
\[u(0,t) > 1-\varepsilon, \qquad \ 0<t<t_1.\] Recall that for each $t$, $v(\cdot,t)$ solves \eqref{v_0solves}, hence $v(\cdot,t)$ depends continuously on $u(\cdot,t)$, and therefore
\[
v(0,t)\to v_0(0) \qquad \text{as } t\to0 .
\]
Thus there exists $t_2>0$ such that
$$|v(0,t)-v_0(0)|<\varepsilon, \qquad 0<t<t_2.$$
Since $v_0(0)<1$, choose $\varepsilon<\frac{1-v_0(0)}{2}$ and set $t_0=\min\{t_1,t_2\}$. Then for $0<t<t_0$,
\[u(0,t) > 1-\varepsilon > v_0(0)+\varepsilon > v(0,t).\]}
\end{remark}

\subsection{Time decay estimates} \label{time.decay.1}

As we have mentioned solutions of the Neumann local heat equation 
$u_t = \Delta u $
converge to the mean value of the initial condition as $t\to \infty$. Our
next goal is to show that solutions to our coupled model reproduce this behavior, at least when $A$ is connected. 
We will prove that for $u_0\in L^2(A)$ with $\int_A u_0\,dx = 0$, $u$, the 
solution to \eqref{localpar}, decays exponentially to zero as $t\to\infty$.  
More precisely, there exists a constant $C>0$ such that
\begin{equation}\label{abneumann.66}
\|u(\cdot,t)\|_{L^2(A)}
\;\le\;
e^{-2C t}\,\|u_0\|_{L^2(A)},
\qquad t\ge 0.
\end{equation}

Moreover, we can make the constant $C$ explicit, in terms of an eigenvalue problem.
Let
$$
\begin{array}{l}
\displaystyle
\lambda_1  = 
\min_{\| u \|_{L^2 (A)=1, \int_A u (x) \, dx =0}} 
 \frac12\int_A |\nabla u(x)|^2\,dx
  + \frac14 \int_B\!\int_B G(x-y)(v(x,t)-v(y,t))^2\,dx\,dy 
  \\[10pt]
  \qquad \qquad \qquad \qquad \displaystyle + \frac12\int_A\int_B J(x-y)\,(v(y)-u(x))^2\,dy\,dx.
  \end{array}
$$
Notice that here we are minimizing in $(u,v)$ with the constraints that 
$\| u \|_{L^2 (A)=1}$ and $\int_A u (x) \, dx =0$.

First, we state an auxiliary lemma

\begin{lemma}\label{lemmaabneumann} Suppose that $A$ is connected. Then,
$$
\lambda_1 > 0.
$$
\end{lemma}

\begin{proof} Arguing by contradiction, suppose that there exists a sequence $u_n\in H^1(A)$, which by Proposition \ref{EUnlp-e} is associated by $v_n\in L^2(B)$ such that
$$\int_A u_n^2(x)\, dx=1, \qquad \int_A u_n (x) \, dx =0,$$
and
$$\int_A |\nabla u_n(x)|^2\, dx \to 0, $$ 
$$\int_A\int_B J(x-y)(u_n(x)-v_n(y))^2\, dx \, dy \to 0,$$
and 
$$\int_B\int_BG(x-y)(v_n(x)-v_n(y))^2\, dx\, dy\to 0,$$
when $n\to \infty$.
Since $\|u_n\|_{L^2(A)}=1$ and 
$$\int_A |\nabla u_n(x)|^2\, dx\to 0,$$
we get that, along a subsequence, $u_n\to k$ strongly in $L^2(A)$ (here we use that $A$ is connected). 
Then, from $\int_A u_n (x) dx=0$ we conclude that $k=0$, which is a contradiction with $$1= \lim_n \int_A u_n^2(x)\, dx =\int_A k^2dx.$$
\end{proof}

Notice that this generates the first eigenvalue of the following problem: In the first region $A$, we consider the Eigenvalue problem of Neumann boundary condition
\begin{equation}\label{eigenvaluelocalp-e}
\begin{cases}
    -\lambda_1\, u=\Delta u(x,t)+ \int_B J(x-y)(v(y,t)-u(x,t))\,dy, &  x\in A\times (0,\infty),  \\
    \displaystyle \frac{\partial u}{\partial \eta}(x,t)=0, & x\in \partial A \times (0,\infty),\\
    u(x,0)=u_0(x), & x\in A,
\end{cases}
\end{equation}
and in $B$,
\begin{equation}\label{e-elliptic}
   0 = \int_B G(x-y)(v(y,t)-v(x,t))\,dy
        +\int_A J(x-y)(u(y,t)-v(x,t))\,dy,
   \qquad x\in B\times (0,\infty).
\end{equation}
To prove that coincides to the weak formulation of eigenvalue, we follow the same argue of above proof that we got \eqref{energy-final}, ie, to obtain the formula 
\begin{align}
-\lambda_1 \int_A &u^2(x)dx=\\
&-\frac{1}{2}\int_A |\nabla u|^2dx-\frac{1}{2}\int_A\int_BJ(x-y)(v(y)-u(x)^2dydx - \int_B\int_BG(x-y)(v(y)-v(x))^2dydx.\notag
\end{align}
To prove that $C_c$ can be optimal by $\lambda_1$, we need to proof the lemma:

\begin{lemma}
There exists $\tilde u\in H^1(A)$ with
\[
\int_A \tilde u(x)\,dx = 0,
\]
such that the infimum in \eqref{eigenvaluelocalp-e} is attained. In particular,
\[
\lambda_1(A,B)=\frac{E(\tilde u)}{\|\tilde u\|_{L^2(A)}}.
\]
\end{lemma}
\begin{proof}
Suppose that $u_n\in H^1(A)$ is a minimizer sequence that satisfies $\displaystyle\int_A u_ndx=0$, $\|u_n\|_{L^2(A)}=1$ and

$$E(u_n)\searrow \lambda_1.$$

By Lemma \eqref{EUnlp-e}, each $u_n\in L^2(A)$ associes to a unique $v_n\in L^2(B)$ that solves weakly \eqref{nleli} and we obtain the estimative \eqref{uvnormp-e} for $C_e>0$
$$ \|v_n\|_{L^2(B)}\le C_e\|u_n\|_{L^2(A)}.$$
Then $\|v_n\|_{L^2(B)}$ is also a bounded sequence.  By the Banach--Alaoglu theorem there exists $v \in L^2(B)$ and $u\in L^2(A)$ and a subsequence, still denoted by $(v_n)$ and $(u_n)$, respectively, such that
\[
v_n \rightharpoonup v \quad \text{weakly in } L^2(B) \qquad \text{and} \qquad u_n \rightharpoonup u \quad \text{weakly in } L^2(A).
\] 

On the other hand, the minimizing sequence satisfies $E(u_n)\leq C$, for some $C>0$, ie

$$C>\frac{1}{2}\int_A |\nabla u_n|^2dx+ \frac{1}{2}\int_A\int_BJ(x-y)(v_n(y)-u_n(x))^2dydx
+\int_B\int_BG(x-y)(v_n(y)-v_n(x))^2dydx.$$

Then, we obtain in particular that

$$C>\frac{1}{2}\int_A |\nabla u_n|^2dx,$$
which implies that $\|\nabla u_n\|_{L^2(A)}$ is bounded then, by Banach--Alaoglu theorem, there exists $w\in L^2(A)$ and a subsequence, still denoted by $\nabla u_n$ such that
$$\nabla u_n  \rightharpoonup w \quad \text{weakly in } H^1(A).$$
Moreover, since $A$ is smooth domain, by Rellich-- Kondrachov ($H^1(A)\hookrightarrow\hookrightarrow L^2(A)$), we have that a subsequence, still denoted by $u_n$ such that

$$u_n\longrightarrow u \qquad \text{strong \ in} \ L^2(A),$$
and $u\in H^1(A)$ since $u_{n}\rightharpoonup u$ in $H^1(A)$ and by definition of weak convergence,
\[
\int_A u_{n_k}\,\partial_i \varphi
= -\int_A \partial_i u_{n_k}\,\varphi
\;\longrightarrow\;
-\int_A w_i\,\varphi, \forall \varphi \in C^\infty (A)
\]
Hence,
\[
-\int_A w_i\,\varphi = \int_A u\,\partial_i \varphi,
\]
which implies
\[
w_i=\partial_i u.
\]
Therefore,
\[
w=\nabla u.
\]
Then, by the lower semicontinuity of the operator, we have

$$\int_A |\nabla u|^2dx\leq \liminf_n\int_A |\nabla u_n|^2dx.$$
Since $v_n\rightharpoonup v$, we also have

$$\int_B\int_BG(x-y)(v(y)-v(x))^2dydx\leq \liminf_n\int_B\int_BG(x-y)(v_n(y)-v_n(x))^2dydx.$$

Finally, by the strong convergence of $u_n\to u$ in $L^2(A)$, we have

$$\int_A\int_BJ(x-y)(v(y)-u(x))^2dydx\leq \liminf_n\int_A\int_BJ(x-y)(v_n(y)-u_n(x))^2dydx.$$

Using the three last inequality, we proved that there exists a $u\in H^1(A)$ that attains the infimum $$E(u)\le \liminf_n E(u_n)=\lambda_1,$$

with $\int_A u(x)dx=0$ and $v\in L^2(A)$ solves weakly the elliptic equation \eqref{nleli}.
By normalization, we get, $\|u\|_{L^2(A)}=1$.
\end{proof}

We now prove Theorem \ref{thm:expdecay.intro}.

\begin{proof}[Proof of Theorem \ref{thm:expdecay.intro}]
Multiplying \eqref{localpar} by $u$ and integrating over $A$, we obtain
\[
\int_A u_t(x,t)u(x,t)\,dx
= \int_A \Delta u(x,t)u(x,t)\,dx
  + \int_A\!\int_B J(x-y)\bigl(v(y,t)-u(x,t)\bigr)u(x,t)\,dy\,dx.
\]
Applying Green's identity and using the Neumann boundary condition
$\frac{\partial u}{\partial\eta}=0$ on $\partial A$, we get
\[
\frac{d}{dt}\int_A \frac{u^2 (x,t)}{2}\,dx
= -\int_A |\nabla u(x,t)|^2\,dx
  + \int_A\!\int_B J(x-y)\bigl(v(y,t)-u(x,t)\bigr)u(x,t)\,dy\,dx.
\label{eq:energy-u}
\]
Next, multiplying \eqref{nleli} by $v(x,t)$ and integrating over $B$ yields
\[
\int_B\!\int_B
G(x-y)\bigl(v(y,t)-v(x,t)\bigr)v(x,t)\,dy\,dx
  + \int_B\!\int_A J(x-y)\bigl(u(y,t)-v(x,t)\bigr)v(x,t)\,dy\,dx
  = 0.
\]
Using the symmetry of $G$ and standard algebraic identities, we obtain
\[
\int_B\!\int_B G(x-y)\bigl(v(y)-v(x)\bigr)v(x)\,dy\,dx
= -\frac12 \int_B\!\int_B G(x-y)(v(x)-v(y))^2\,dx\,dy,
\]
and
\[
\int_A\!\int_B J(x-y)\bigl(u(x)-v(y)\bigr)v(y)\,dx\,dy
=
-\int_A\!\int_B J(x-y)v(y)^2\,dx\,dy
+ \int_A\!\int_B J(x-y)u(x)v(y)\,dx\,dy.
\]

Adding this identity to \eqref{eq:energy-u} gives
\begin{align*}
\frac{d}{dt}\int_A \frac{u^2 (x,t)}{2}\,dx
&= -\int_A |\nabla u(x,t)|^2\,dx
   - \frac12 \int_B\!\int_B G(x-y)(v(x,t)-v(y,t))^2\,dx\,dy \\
&\quad
   -\int_A\!\int_B J(x-y)v(y,t)^2\,dy\,dx
   -\int_A\!\int_B J(x-y)u(x,t)^2\,dy\,dx \\
&\quad
   + 2\int_A\!\int_B J(x-y)u(x,t)v(y,t)\,dy\,dx.
\end{align*}

Observe that the last three terms combine exactly into
\[
-\int_A\!\int_B J(x-y)\bigl(u(x,t)-v(y,t)\bigr)^2\,dy\,dx.
\]
Therefore,
\begin{align}
\frac{d}{dt}\int_A \frac{u^2(x,t)}{2}\,dx
&= -\int_A |\nabla u(x,t)|^2\,dx
   -\frac12 \int_B\!\int_B G(x-y)(v(x,t)-v(y,t))^2\,dx\,dy \notag\\
&\quad
   -\int_A\!\int_B J(x-y)\bigl(u(x,t)-v(y,t)\bigr)^2\,dy\,dx.
\label{energy-final}
\end{align}
By  Lemma~\ref{lemmaabneumann}, the right-hand side is bounded above by
\[
- \frac{\lambda_1}2 \int_A u^2(x,t)\,dx.
\]
Then, we have the differential
inequality
\[
\frac{d}{dt}\int_A \frac{u^2(x,t)}{2}\,dx
\le - \frac{\lambda_1}2 \int_A u^2(x,t)\,dx,
\]
and we conclude that
\[
\int_A u^2(x,t)\,dx
\le e^{- \lambda_1 t}\|u_0\|_{L^2(A)}^2,
\]
which proves the desired exponential decay.
\end{proof}

\subsection{The parabolic-elliptic model can be obtained as a limit of a parabolic system}
\label{subsect-limit-1}

For fixed initial data  $u_0\in L^2(A)$ and $v_0\in L^2(B)$, we introduced in the introduction the $\varepsilon$-family $(u^\varepsilon,v^\varepsilon)$ of solutions to the following $\varepsilon$--problem. Let us recall it here for the reader's convenience,
\begin{equation}\label{elocalpar}
\begin{cases}
u_t^\varepsilon(x,t)
   = \Delta u^\varepsilon(x,t)
     + \displaystyle\int_B 
         J(x-y)\big( v^\varepsilon(y,t) - u^\varepsilon(x,t) \big)\,dy,
   & x\in A,\ t\in(0,T],\\[1.1em]
\displaystyle \frac{\partial u^\varepsilon}{\partial \eta}(x,t)=0,
   & x\in\partial A,\ t\in[0,T],\\[0.7em]
u^\varepsilon(x,0)=u_0(x),
   & x\in A,
\end{cases}
\end{equation}
and in $B$,
\begin{equation}\label{enleli}
\begin{cases}
\varepsilon\, v_t^\varepsilon(x,t)
   = \displaystyle\int_B 
        G(x-y)\big( v^\varepsilon(y,t) - v^\varepsilon(x,t) \big)\,dy
     + \displaystyle\int_A 
        J(x-y)\big( u^\varepsilon(y,t) - v^\varepsilon(x,t) \big)\,dy,
   & x\in B,\ t\in(0,T],\\[1.1em]
v^\varepsilon(x,0)=v_0(x),
   & x\in B.
\end{cases}
\end{equation}
In this formulation, the variable $v^\varepsilon$ evolves with a
relaxation time of order $\mathcal{O}(\varepsilon)$, meaning that its
transient dynamics is very fast compared with the time scale of $u^\varepsilon$. The uniform estimates yield compactness: up to a subsequence, we obtain
\[ u^\varepsilon\rightharpoonup u \quad\text{weakly in } L^2 ((0,T),H^1(A)), \qquad v^\varepsilon (\cdot, t)\rightharpoonup v (\cdot, t) \quad\text{weakly in }L^2(B). \]
Passing to the limit in the weak formulation of \eqref{elocalpar}--\eqref{enleli} then reveals the structure of the limit system. In the equation for $u^\varepsilon$, the limit is straightforward, since no singular term is present. In the equation for $v^\varepsilon$, the factor $\varepsilon$ forces the term $\varepsilon v^\varepsilon_t$ to vanish in the limit, since $v_t^\varepsilon$ is bounded, recovering precisely the stationary nonlocal equation \eqref{nleli}.  Thus the limit $(u,v)$ solves the original parabolic--elliptic system \eqref{localpar}--\eqref{nleli}.

\begin{proof}[Proof of Theorem \ref{approximation.intro}]
We begin with equation \eqref{elocalpar}. Multiplying it by 
$u^\varepsilon(x,t)$ and integrating over $A$, we obtain
\begin{align*}
\int_A u_t^\varepsilon(x,t)\,u^\varepsilon(x,t)\,dx
&=
\int_A \Delta u^\varepsilon(x,t)\,u^\varepsilon(x,t)\,dx
  + \iint_{A\times B}
    J(x-y)\big( v^\varepsilon(y,t)-u^\varepsilon(x,t) \big)
    u^\varepsilon(x,t)\,dy\,dx.
\end{align*}
Using Green's identity together with the Neumann condition 
$\frac{\partial u^\varepsilon}{\partial \eta} =0$ on $\partial A$, we get
\[
\int_A u_t^\varepsilon(x,t)\,u^\varepsilon(x,t)\,dx
=
-\int_A |\nabla u^\varepsilon(x,t)|^2\,dx
+\iint_{A\times B}
  J(x-y)\big( v^\varepsilon(y,t)-u^\varepsilon(x,t) \big)
  u^\varepsilon(x,t)\,dy\,dx.
\]
Hence,
\begin{equation}\label{eq:int-u-eps}
\frac{d}{dt}\,\frac12\!\int_A |u^\varepsilon(x,t)|^2 dx
=
-\int_A |\nabla u^\varepsilon(x,t)|^2\,dx
+\iint_{A\times B} J(x-y)
   (v^\varepsilon(y,t)-u^\varepsilon(x,t))u^\varepsilon(x,t)\,dy\,dx.
\end{equation}
Integrating \eqref{eq:int-u-eps} in time in $[0,t_0]$ for some $t_0\leq T$ gives
\begin{align}\label{intelocalpar}
\frac12\!\int_A |u^\varepsilon(x,{t_0})|^2 dx
-\frac12\!\int_A |u_0(x)|^2 dx
&=
-\int_0^{t_0}\!\!\int_A |\nabla u^\varepsilon(x,t)|^2\,dx\,dt
\notag\\
&\quad
+\int_0^{t_0}\!\!\iint_{A\times B}
  J(x-y)\big( v^\varepsilon(y,t)-u^\varepsilon(x,t) \big)u^\varepsilon(x,t)
\,dy\,dx\,dt.
\end{align}

Now, multiply \eqref{enleli} by $v^\varepsilon(x,t)$ and integrate over $B$ to obtain,
\begin{align*}
\varepsilon\!\int_B v_t^\varepsilon(x,t)\,v^\varepsilon(x,t)\,dx
&=
\iint_{B\times B}
 G(x-y)\big( v^\varepsilon(y,t)-v^\varepsilon(x,t) \big)
 v^\varepsilon(x,t)\,dy\,dx
\\
&\quad
+\iint_{A\times B}
 J(x-y)\big( u^\varepsilon(y,t)-v^\varepsilon(x,t) \big)
 v^\varepsilon(x,t)\,dy\,dx.
\end{align*}
Using the symmetry of $G$ and integrating in time we get
\begin{align}\label{intenlelip}
\frac{\varepsilon}{2}\!\int_B |v^\varepsilon(x,{t_0})|^2\,dx
-\frac{\varepsilon}{2}\!\int_B |v_0(x)|^2\,dx
&=
-\int_0^{t_0} \iint_{B\times B}
 G(x-y)\big( v^\varepsilon(y,t)-v^\varepsilon(x,t) \big)^2
\,dy\,dx\,dt
\notag\\
&\quad
+\int_0^{t_0}\!\!\iint_{A\times B}
 J(x-y)\big( u^\varepsilon(y,t)-v^\varepsilon(x,t) \big)
 v^\varepsilon(x,t)\,dy\,dx\,dt.
\end{align}
Adding \eqref{intelocalpar} and \eqref{intenlelip}, we obtain
\begin{align*}
&\frac12\!\int_A |u^\varepsilon(x,{t_0})|^2\,dx
+\frac{\varepsilon}{2}\!\int_B |v^\varepsilon(x,{t_0})|^2\,dx
-\frac12\!\int_A |u_0(x)|^2\,dx
-\frac{\varepsilon}{2}\!\int_B |v_0(x)|^2\,dx
\\[0.4em]
&=
-\int_0^{t_0}\!\!\int_A |\nabla u^\varepsilon(x,t)|^2\,dx\,dt
-\int_0^{t_0}\!\!\iint_{B\times B}
   G(x-y)\big( v^\varepsilon(y,t)-v^\varepsilon(x,t) \big)^2
\,dy\,dx\,dt
\\
&\quad
+\int_0^{t_0}\!\!\iint_{A\times B}
 J(x-y)\big( v^\varepsilon(y,t)-u^\varepsilon(x,t) \big)
 \big( u^\varepsilon(x,t)-v^\varepsilon(x,t) \big)
\,dy\,dx\,dt.
\end{align*}

Notice that, by Lemma \ref{lemmaabneumann}, we get
\begin{align*}
&\frac12\!\int_A |u^\varepsilon(x,{t_0})|^2\,dx
\;+\;
\frac{\varepsilon}{2}\!\int_B |v^\varepsilon(x,{t_0})|^2\,dx
\;-\;
\frac12\!\int_A |u_0(x)|^2\,dx
\;-\;
\frac{\varepsilon}{2}\!\int_B |v_0(x)|^2\,dx
\\[0.6em]
&=
-\int_{0}^{{t_0}}\!\!\int_A |\nabla u^\varepsilon(x,t)|^2\,dx\,dt
\;-\;
\int_{0}^{{t_0}}\!\!\iint_{B\times B}
   G(x-y)\,\big( v^\varepsilon(y,t)-v^\varepsilon(x,t) \big)^2
\,dy\,dx\,dt
\\
&\qquad
+\int_{0}^{{t_0}}\!\!\iint_{A\times B}
   J(x-y)\,\big( v^\varepsilon(y,t)-u^\varepsilon(x,t) \big)
           \big( u^\varepsilon(x,t)-v^\varepsilon(x,t) \big)
\,dy\,dx\,dt
\\[0.6em]
&=
-\int_{0}^{{t_0}}\!\!\int_A |\nabla u^\varepsilon(x,t)|^2\,dx\,dt
\;-\;
\int_{0}^{{t_0}}\!\!\iint_{B\times B}
   G(x-y)\,\big( v^\varepsilon(y,t)-v^\varepsilon(x,t) \big)^2
\,dy\,dx\,dt
\\
&\qquad
-\int_{0}^{{t_0}}\!\!\iint_{A\times B}
   J(x-y)\,\big( v^\varepsilon(y,t)-u^\varepsilon(x,t) \big)^2
\,dy\,dx\,dt\leq -C_c\int_0^{t_0}\int_A (u^\epsilon)^2(x,t)dx.
\end{align*}
Since $\epsilon>0$ is small, we get for some $\Tilde{M}>0$ that
\begin{align*}
     \Tilde{M}\geq &
\frac12\!\int_A |u_0(x)|^2\,dx
\;+\;
\frac{\varepsilon}{2}\!\int_B |v_0(x)|^2\,dx \\ & \geq  C_c\int_0^{t_0}\int_A (u^\epsilon)^2(x,t)dxdt + \frac12\!\int_A |u^\varepsilon(x,{t_0})|^2\,dx
+
\frac{\varepsilon}{2}\!\int_B |v^\varepsilon(x,{t_0})|^2\,dx.
\end{align*}
Hence, we deduce that
\[ 
\|u^\varepsilon(\cdot,t_0)\|_{L^2(A)}^2
\quad\text{is uniformly bounded in }\varepsilon>0, \ \text{for any} \ t_0 \in [0,T].
\]
Then, by \eqref{vunormp-e},
\[
\|v^\varepsilon(\cdot,t_0)\|_{L^2(B)}^2
\quad\text{is also uniformly bounded in }\varepsilon>0.\ \text{for any} \ t_0 \in [0,T].
\]
Thus, there exist $u\in L^\infty ((0,T), L^2(A))$ and $v\in L^\infty ((0,T), L^2(B))$ such that, up to a subsequence,
\[
u^\varepsilon \rightharpoonup u \quad\text{in }L^2(A), \ \text{for almost any} \ t_0 \in [0,T], 
\qquad
v^\varepsilon \rightharpoonup v \quad\text{in }L^2(B), \ \text{for almost any} \ t_0 \in [0,T].
\]
\medskip
Next, using \eqref{intelocalpar}, we have
\begin{align*}
    &-\int_A \frac{(u_0)^2(x)}{2}dx\leq \int_A \frac{(u^\epsilon_t)^2(x,T)}{2}dx-\int_A \frac{(u^\epsilon_t)^2(x,0)}{2}dx=\notag \\
    &-\int_{0}^{T}\int_A|\nabla u^\epsilon(x,t)|^2dxdt + \int_{0}^{T}\iint_{A\times B} J(x-y)v^\epsilon(y,t)u^\epsilon(x,t)-\int_{0}^{T}\iint_{A\times B} J(x-y)(u^\epsilon)^2(x,t)dydxdt\\
&\leq -\int_{0}^{T}\int_A|\nabla u^\epsilon(x,t)|^2dxdt +C_J|A||B|\|u^\epsilon\|_{L^2(A)}\|v^\epsilon\|_{L^2(B)},
\end{align*}
concluding that the gradient term is uniformly bounded, since $\|u^\varepsilon\|_{L^2(A)}$ and $\|v^\varepsilon\|_{L^2(B)}$ are bounded. 
$$\int_0^T\!\!\int_A |\nabla u^\varepsilon|^2\,dx\,dt\le C.$$
Therefore,
\begin{equation}\label{eH1weakconvp-e}
u^\varepsilon \rightharpoonup u
\qquad\text{weakly in } L^2 ((0,T),H^1(A)).
\end{equation}

We now pass to the limit in the weak formulations.  
Let $\varphi\in C^1([0,T];H^1(A))$ with $\varphi(\cdot,T)=0$.  
Multiplying \eqref{elocalpar} by $\varphi$ and integrating in $[0,T]$ we obtain
\begin{align}\label{elocalp-e}
\int_A\!\Big( u^\varepsilon(x,T)\underbrace{\varphi(x,T)}_{=0}
      -u_0(x)\varphi(x,0)\Big)\,dx
-&\!\int_0^T\!\!\int_A u^\varepsilon\,\varphi_t\,dx\,dt
=
-\!\int_0^T\!\!\int_A \nabla u^\varepsilon\cdot\nabla\varphi\,dx\,dt
\notag\\
&\quad
+\!\int_0^T\!\!\iint_{A\times B}
  J(x-y)\big( v^\varepsilon(y,t)-u^\varepsilon(x,t) \big)
  \varphi(x,t)\,dy\,dx\,dt.
\end{align}

Similarly, letting $\psi\in C^1([0,T];L^2(B))$ with $\psi(\cdot,T)=0$,
multiplying \eqref{enleli} by $\psi$ and integrating gives
\begin{align}\label{enlp-e}
-\varepsilon\!\int_B v_0(x)\psi(x,0)\,dx
-\varepsilon\!\int_0^T\!\!\int_B v^\varepsilon\,\psi_t\,dx\,dt
&=
\int_0^T\!\!\iint_{B\times B}
  G(x-y)\big( v^\varepsilon(y,t)-v^\varepsilon(x,t) \big)
  \psi(x,t)\,dy\,dx\,dt
\notag\\
&\quad
+\int_0^T\!\!\iint_{A\times B}
  J(x-y)\big( u^\varepsilon(y,t)-v^\varepsilon(x,t) \big)
  \psi(x,t)\,dy\,dx\,dt.
\end{align}

Since $\varphi\in C^1([0,T];H^1(A))$ and $\psi\in C^1([0,T];L^2(B))$, 
passing to the limit as $\varepsilon\to0$ in \eqref{elocalp-e} yields
\begin{align*}
\int_0^T\!\!\int_A u^\varepsilon(x,t)\,\varphi_t(x,t)\,dx\,dt
&\longrightarrow
\int_0^T\!\!\int_A u(x,t)\,\varphi_t(x,t)\,dx\,dt,\\[0.2em]
\int_0^T\!\!\int_B v^\varepsilon(x,t)\,\psi_t(x,t)\,dx\,dt
&\longrightarrow
\int_0^T\!\!\int_B v(x,t)\,\psi_t(x,t)\,dx\,dt.
\end{align*}
For the nonlocal terms, observe that for each fixed $x$ and $t$ the function
\[
y\mapsto J(x-y)\,v^\varepsilon(y,t)
\]
belongs to $L^2(B)$, and the convolution of a continuous function with an $L^2$-function 
remains continuous. Hence, for each $(x,t)$, we have that
\[
\int_B J(x-y)v^\varepsilon(y,t)\,dy \;\longrightarrow\;
\int_B J(x-y)v(y,t)\,dy.
\]
Therefore,
\begin{align*}
&\lim_{\varepsilon\to0}
\int_0^T\!\!\int_A 
\varphi(x,t)\left(\int_B J(x-y)\,v^\varepsilon(y,t)\,dy\right)\,dx\,dt
-
\lim_{\varepsilon\to0}
\int_0^T\!\!\int_A 
u^\varepsilon(x,t)\left(\int_B J(x-y)\,dy\right)\varphi(x,t)\,dx\,dt\\[0.4em]
&=
\int_0^T\!\!\int_A 
\varphi(x,t)\left(\int_B J(x-y)\,v(y,t)\,dy\right)\,dx\,dt
-
\int_0^T\!\!\int_A 
u(x,t)\left(\int_B J(x-y)\,dy\right)\varphi(x,t)\,dx\,dt.
\end{align*}

For the other nonlocal terms in \eqref{enlp-e}, the argument is similar and, finally, we observe that we can pass to the limit
in the term involving $\nabla u^\varepsilon$ thanks to \eqref{eH1weakconvp-e}. 

Then, passing to
the limit $\varepsilon\to 0$, we obtain
\begin{align}
&-\int_A u_0(x)\varphi(x,0)\,dx
-\int_0^T\!\!\int_A u\,\varphi_t\,dx\,dt
\notag\\
&=
-\int_0^T\!\!\int_A \nabla u\cdot\nabla\varphi\,dx\,dt
+\int_0^T\!\!\iint_{A\times B}
   J(x-y)\big( v(y,t)-u(x,t) \big)\varphi(x,t)\,dy\,dx\,dt
\notag\\
&\quad
+\int_0^T\!\!\iint_{B\times B}
   G(x-y)\big( v(y,t)-v(x,t) \big)\psi(x,t)\,dy\,dx\,dt
\notag\\
&\quad
+\int_0^T\!\!\iint_{A\times B}
   J(x-y)\big( u(y,t)-v(x,t) \big)\psi(x,t)\,dy\,dx\,dt.
\end{align}

Since this holds for every test functions
$\varphi\in C^1([0,T];H^1(A))$ with $\varphi(\cdot,T)=0$ and
$\psi\in C^1([0,T];L^2(B))$ with $\psi(\cdot,T)=0$, we conclude that
$(u,v)$ satisfies the weak formulation of \eqref{localpar}--\eqref{nleli},
as we wanted to show.
\end{proof}

\section{The elliptic-parabolic model} \label{sect-2-model}

In this section, we gather the main ideas needed to extend our results to the case in which the local part of the system is elliptic and the nonlocal part parabolic. We just 
skecht the details for the existence and uniqueness of solutions to \eqref{elilocal}-\eqref{parnl} as the main idea runs as in the previous case. 
For each $v\in C([0,T],L^2(B))$, we obtain, by the Lax-Milgram Theorem, a unique solution $u\in C([0,T],H^1(A))$ of \eqref{elilocal}. This defines a well-defined Lipschitz linear operator $S_{BA}:C([0,T],L^2(B))\to C([0,T],H^1(A))$. To deal with \eqref{parnl}, we use the Banach Fixed Point. Given $u\in C([0,T],H^1(A))$, there exists a unique solution $v\in C([0,T],L^2(B))$ of \eqref{parnl} and this defines a well-defined Lipschitz linear operator $S_{AB}:C([0,T],H^1(A))\to C([0,T],L^2(B))$. At the end, we have a contractive operator $S$, which is given in terms of $S_{AB},S_{BA}$, whose fixed point gives the unique solution of the system \eqref{elilocal}--\eqref{parnl} for all $t>0$.

\begin{lemma}\label{EUlocaleli}
There exists a linear operator
\[
S_{BA}:C([0,T],L^2(B))\to C([0,T],H^1(A))
\]
such that, for every $v\in L^2(B)$, the function $u=S_{BA}(v)$ is the unique
solution of \eqref{elilocal}. Moreover, the operator $S_{BA}$ is continuous and
Lipschitz, that is, there exists a positive constant
$C_{BA}=C(A,B,J)$ such that, for all $v_1,v_2\in L^2(B)$,
\begin{equation}\label{uvnorme-p}
\|S_{BA}(v_1)(\cdot, t)-S_{BA}(v_2)(\cdot, t)\|_{L^2(A)}
\le
C_{BA}\,\|v_1(\cdot, t)-v_2(\cdot, t)\|_{L^2(B)}.
\end{equation}
\end{lemma}

\begin{proof}[Proof]
Consider the bilinear form $b:H^1(A)\times H^1(A)\to\mathbb{R}$ defined by
\[
b(u,\varphi)
=\int_A \nabla u(x)\cdot\nabla\varphi(x)\,dx
+\int_A\int_B J(x-y)\,dy\,u(x)\varphi(x)\,dx .
\]
By the Cauchy--Schwarz inequality and the boundedness of $J$ by a
constant $C_J>0$, we obtain that 
\[
|b(u,\varphi)|
\le \|\nabla u\|_{L^2(A)}\|\nabla\varphi\|_{L^2(A)}
+ C_J \|u\|_{L^2(A)}\|\varphi\|_{L^2(A)}
\le C\,\|u\|_{H^1(A)}\|\varphi\|_{H^1(A)},
\]
for all $u,\varphi\in H^1(A)$.
Hence $b$ is continuous. Moreover, it is easy to see that $b$ is coercive on $H^1(A)$.

For $v\in L^2(B)$, define the linear functional $f_v:H^1(A)\to\mathbb{R}$ by
\[
f_v(\varphi)
=\int_A\int_B J(x-y)\,v(y)\varphi(x)\,dy\,dx .
\]
Using again the boundedness of $J$, we obtain
\[
|f_v(\varphi)|
\le C_J \|v\|_{L^2(B)}\|\varphi\|_{L^2(A)}
\le C\,\|v\|_{L^2(B)}\|\varphi\|_{H^1(A)},
\]
so $f_v$ is continuous. Therefore, by the Lax--Milgram theorem, there exists a
unique $u\in H^1(A)$ such that
\[
b(u,\varphi)=f_v(\varphi),
\qquad \forall\,\varphi\in H^1(A),
\]
which is the weak formulation of \eqref{elilocal} and this proves that $S_{BA}$ is well defined.

Now let $v_1,v_2\in L^2(B)$ and denote by $u_1,u_2\in H^1(A)$ the corresponding
solutions. Subtracting the equations satisfied by $u_1$ and $u_2$, testing the
resulting equation with $u_1-u_2$, and integrating over $A$, we obtain
\[
\int_A |\nabla(u_1-u_2)|^2\,dx
+\int_A\int_B J(x-y)\,(u_1-u_2)^2(x)\,dy\,dx
=\int_A\int_B J(x-y)\,(v_1-v_2)(y)(u_1-u_2)(x)\,dy\,dx .
\]
Since the function
\[
b(x):=\int_B J(x-y)\,dy,\qquad x\in A,
\]
is strictly positive in some subdomian of $A$, there exists $c>0$ such that
\[
\int_A |\nabla(u_1-u_2)|^2\,dx
+\int_A\int_B J(x-y)\,(u_1-u_2)^2(x)\,dy\,dx
\ge c\,\|u_1-u_2\|_{L^2(A)}^2.
\]
Applying the Cauchy--Schwarz inequality to the right-hand side and absorbing
terms, we conclude that
\[
\|u_1(\cdot, t)-u_2(\cdot, t)|_{L^2(A)}
\le C_{BA}\,\|v_1(\cdot, t)-v_2(\cdot, t)\|_{L^2(B)},
\]
for some constant $C_{BA}=C(A,B,J)>0$, which proves \eqref{uvnorme-p}.
\end{proof}

\begin{proposition}\label{EUnlparab}
Given $u\in C([0,T];H^1(A))$ and $v_0\in L^2(B)$, there exists a unique $v\in C([0,T];L^2(B))$ solution to \eqref{parnl}. 
Moreover, if $v_1,v_2$
are the solutions corresponding to $u_1,u_2$, then if
\[
T\in\Bigg(0,\frac{1}{C_G + C_{G^2}^{1/2}+C_J|A|^{1/2}}\Bigg)
\]
there holds the following estimate 
\begin{equation}\label{vunorme-p}
\|v_1-v_2\|_{B}
\le
\frac{C_JT}{1-T(C_G + C_{G^2}^{1/2}+C_J|A|^{1/2})}\,
\|u_1-u_2\|_{A}.
\end{equation}
where $\|u\|_A=\sup_{t\in [0,T]}\|u(\cdot,t)\|_{L^2(A)}$ and $\|v\|_B=\sup_{t\in [0,T]}\|v(\cdot,t)\|_{L^2(B)}$.
\end{proposition}

\begin{proof}[Proof]
Fix $u\in C([0,T];H^1(A))$ and rewrite \eqref{parnl} in its integral form. 
This motivates the definition of the operator 
$S_u:C([0,T];L^2(B))\to C([0,T];L^2(B))$ given by
\[
S_u(v)(x,t)
=
v_0(x)
+\int_0^t\!\int_B G(x-y)\big(v(y,s)-v(x,s)\big)\,dy\,ds
+\int_0^t\!\int_A J(x-y)\big(u(y,s)-v(x,s)\big)\,dy\,ds.
\]

Let $v_1,v_2\in C([0,T];L^2(B))$. Subtracting the corresponding integral
representations, the initial datum cancels and we obtain an expression
depending only on $v_1-v_2$. Using the triangle inequality, Fubini's theorem,
and the boundedness of the kernels, we estimate
\[
\|S_u(v_1)(\cdot,t)-S_u(v_2)(\cdot,t)\|_{L^2(B)}
\le
\int_0^t \Big(
\mathcal{I}_G(s)+\mathcal{I}_J(s)
\Big)\,ds,
\]
where $\mathcal I_G(s)$ is defined by
\[
\mathcal I_G(s)
:=
\left\|
\int_B G(\cdot-y)\big[(v_1-v_2)(y,s)-(v_1-v_2)(\cdot,s)\big]\,dy
\right\|_{L^2(B)},
\]
while $\mathcal I_J(s)$ is given by
\[
\mathcal I_J(s)
:=
\left\|
\int_A J(\cdot-y)\big[(v_2-v_1)(\cdot,s)\big]\,dy
\right\|_{L^2(B)}.
\]
Concerning $\mathcal{I}_G$, the symmetry and boundedness of $G$ gives us
\[
\mathcal I_G(s)
=
\left\|
\int_B G(\cdot-y)\big(w(y,s)-w(\cdot,s)\big)\,dy
\right\|_{L^2(B)}.
\]
where we used that
\[
\int_B G(x-y)\big(w(y,s)-w(x,s)\big)\,dy
=
\int_B G(x-y)w(y,s)\,dy
-
w(x,s)\int_B G(x-y)\,dy.
\]
Since the kernel $G$ is nonnegative and bounded, we have
\[
\int_B G(x-y)\,dy \le C_G,
\]
uniformly for all $x$, and therefore
\[
\left\|
w(\cdot,s)\int_B G(\cdot-y)\,dy
\right\|_{L^2(B)}
\le
C_G\,\|w(\cdot,s)\|_{L^2(B)}.
\]
For the other term, we apply Cauchy--Schwarz inequality in the variable $y$ to obtain
\[
\left|
\int_B G(x-y)w(y,s)\,dy
\right|
\le
\left(\int_B G(x-y)^2\,dy\right)^{1/2}
\|w(\cdot,s)\|_{L^2(B)}.
\]
This yields
\[
\left\|
\int_B G(\cdot-y)w(y,s)\,dy
\right\|_{L^2(B)}
\le
C_{G^2}^{1/2}\,\|w(\cdot,s)\|_{L^2(B)},
\]
for some constant $C_{G^2}$.

Combining the estimates, we conclude that
\[
\mathcal I_G(s)
\le
\big(C_G + C_{G^2}^{1/2}\big)\,
\|v_1(\cdot,s)-v_2(\cdot,s)\|_{L^2(B)}.
\]
Similarly, for the coupling term, the boundedness of $J$
together with the finiteness of $|A|$ implies
\[
\mathcal{I}_J(s)
\le
C_J|A|^{1/2}\|v_1(\cdot,s)-v_2(\cdot,s)\|_{L^2(B)}.
\]
From these estimates we obtain
\[
\|S_u(v_1)(\cdot,t)-S_u(v_2)(\cdot,t)\|_{L^2(B)}
\le
\int_0^t
\big(C_G + C_{G^2}^{1/2}+C_J|A|^{1/2}\big)
\|v_1(\cdot,s)-v_2(\cdot,s)\|_{L^2(B)}\,ds.
\]
Taking the supremum with respect to $t\in[0,T]$ and applying Cauchy--Schwarz in time, we conclude that
\[
\|S_u(v_1)-S_u(v_2)\|_{B}
\le
T\big(C_G + C_{G^2}^{1/2}+C_J|A|^{1/2}\big)
\|v_1-v_2\|_{B}.
\]
Therefore, if
\[
T<\frac{1}{(C_G + C_{G^2}^{1/2}+C_J|A|^{1/2})},
\]
the operator $S_{u}$ is a strict contraction on $C([0,T];L^2(B))$ By Banach's
fixed point theorem, there exists a unique fixed point
$v\in C([0,T];L^2(B))$, which is the unique solution of \eqref{parnl}
corresponding to the fixed function $u$ and it is well defined the operator $S_{AB}$.

The Lipschitz estimate \eqref{vunorme-p} follows by repeating the above
argument for two different functions $u_1,u_2\in C([0,T];H^1(A))$ ((the proof runs as the one for Proposition \eqref{EULocalp-e})) 
and using that the only additional contribution arises from the coupling term, which is
controlled by $C_JT\|u_1-u_2\|_A$. This concludes the proof.
\end{proof}

Now, we are able to prove that the system \eqref{elilocal}--\eqref{parnl} has a unique solution. 

\begin{theorem}
Given $v_0\in L^2(B)$, there exists a unique solution $(u,v)$ to the coupled
system \eqref{elilocal}--\eqref{parnl} with initial datum $v_0$.
\end{theorem}
\begin{proof}[Proof]
Define the operator
\[
S:C([0,T];L^2(B))\to C([0,T];L^2(B)),
\qquad
S:=S_{AB}\circ S_{BA},
\]
where $S_{BA}$ is defined in Lemma~\ref{EUlocaleli} and $S_{AB}$ in
Lemma~\ref{EUnlparab}.  

Let $v_1,v_2\in C([0,T];L^2(B))$ and set $\tilde v_i=S(v_i)$, $i=1,2$.
Combining the Lipschitz estimates for $S_{BA}$ and $S_{AB}$, we obtain
\[
\|\tilde v_1-\tilde v_2\|_B
\le
\underbrace{
\frac{C_JT}{1-T(C_G + C_{G^2}^{1/2}+C_J|A|^{1/2})}
\cdot
C_{AB}
}_{=:K(T)}
\|v_1-v_2\|_B.
\]
Choosing $T>0$ sufficiently small so that $K(T)<1$, the operator $S$ is a
strict contraction on $C([0,T];L^2(B))$. Therefore, by the Banach Fixed Point
Theorem, $S$ admits a unique fixed point, which yields the unique solution of the coupled system \eqref{elilocal}--\eqref{parnl}. To obtain the result for all $t>0$, we can extend the same result for $[T,2T]$, and by iteration, extend for all $T>0$.
\end{proof}

The proof of Theorem \ref{teo.asymp.2} follows the same lines as in subsection \ref{time.decay.1}.
In fact, arguing as before, taking the difference $v (x,t) - \overline{v_0}$ with $\overline{v_0} := \frac{1}{|B|} \int_B v_0(x)\,dx$, we only need to show that for
an initial datum $v_0\in L^2(B)$ with $\int_B v_0\,dx = 0$, 
there exists a constant $C>0$ such that
$$
\| v (\cdot,t)\|_{L^2(B)}
\;\le\;
e^{-2C t}\,\|v_0\|_{L^2(B)},
\qquad t\ge 0.
$$
The constant $C$ is given by
$$
\begin{array}{l}
\displaystyle
\lambda_1  = 
\min_{\| v \|_{L^2 (B)=1}, \int_B v (x) \, dx =0} 
 \frac12\int_A |\nabla u(x)|^2\,dx
  + \frac14 \int_B\!\int_B G(x-y)(v(x)-v(y))^2\,dx\,dy 
  \\[10pt]
  \qquad \qquad \qquad \qquad \displaystyle + \frac12\int_A\int_B J(x-y)\,(v(y)-u(x))^2\,dy\,dx.
  \end{array}
$$

In fact, we have that $\lambda_1$ is strictly positive.

\begin{lemma}\label{lemmaabneumann.2} Suppose that $B$ is connected. Then,
$$
\lambda_1 > 0.
$$
\end{lemma}

\begin{proof}
It follows by similar arguments to those used in the proof of Lemma \ref{lemmaabneumann}.
\end{proof}

We now are ready to prove Theorem \ref{teo.asymp.2}.

\begin{proof}[Proof of Theorem \ref{teo.asymp.2}]
Multiplying \eqref{elilocal} by $v$, \eqref{parnl} by $u$, and integrating over $B$ and $A$ respectively, after some computations, we obtain
\begin{align}
\frac{d}{dt}\int_B \frac{v^2(x,t)}{2}\,dx
&= -\int_A |\nabla u(x,t)|^2\,dx
   -\frac12 \int_B\!\int_B G(x-y)(v(x,t)-v(y,t))^2\,dx\,dy \notag\\
&\quad
   -\int_A\!\int_B J(x-y)\bigl(u(x,t)-v(y,t)\bigr)^2\,dy\,dx.
\label{energy-final.66} \notag
\end{align}
Now, the right-hand side is bounded above by
\[
- \frac{\lambda_1}2 \int_B v^2(x,t)\,dx.
\]
Then, we have the differential
inequality
\[
\frac{d}{dt}\int_B \frac{v^2(x,t)}{2}\,dx
\le - \frac{\lambda_1}2 \int_B v^2(x,t)\,dx,
\]
and we conclude that
\[
\int_B v^2(x,t)\,dx
\le e^{- \lambda_1 t}\|v_0\|_{L^2(A)}^2,
\]
and the desired exponential decay is proved.
\end{proof}

Finally, we observe that this elliptic-parabolic system can be obtained as a limit
of a parabolic problem.

\begin{proof}[Proof of Theorem \ref{approximation}]
The proof follows the same ideas used in subsection \ref{subsect-limit-1} therefore we leave the details to the reader.  
\end{proof}

%%%%%%%%%%%%%%%%%%%%%%%%%%%%%%%%%%%%%%%%%%%%%%%%%%%%%%
	%          7. Acknowlegments
	%%%%%%%%%%%%%%%%%%%%%%%%%%%%%%%%%%%%%%%%%%%%%%%%%%%%%%
\section*{Acknowledgments}
	
	We want to thank M. Pereira for several nice discussions that helped us to improve our results.
	
	This work was started during a visit of L. C. Rosa da Silva to Univ. Torcuato di Tella, and continued with a visit of 
	J. D. Rossi to IME, Unv. Sao Paulo The authors are grateful to both
	institutions for the friendly and stimulating working atmosphere.

	 J. D. Rossi was partially supported by 
			CONICET PIP GI No 11220150100036CO
(Argentina), PICT-03183 (Argentina) and UBACyT 20020160100155BA (Argentina).


\begin{thebibliography}{99}

\bibitem[Acosta et al., 2022]{acosta}
Acosta, G.; Bersetche, F.; Rossi, J.~D.
Local and nonlocal energy-based coupling models.
\emph{SIAM J. Math. Anal.} 54 (2022), no.~6, 6288--6322.

\bibitem[Andreu-Vaillo et al., 2010]{julio}
Andreu-Vaillo, F.; Maz\'on, J. M.; Rossi, J. D.; Toledo-Melero, J. J.
\emph{Nonlocal diffusion problems}.
Mathematical Surveys and Monographs, vol.~165. AMS, 2010.

\bibitem[Azdoud et al., 2013]{Peri1}
Azdoud, Y.; Han, F.; Lubineau, G.
{A morphing framework to couple non-local and local anisotropic continua}.
{\it Inter. J. Solids Structures.} 50(9), (2013), 1332--1341.

\bibitem[Badia et al., 2007]{Peri2}
Badia, S.; Bochev, P.; Lehoucq, R.; Parks, M.; Fish, J.; Nuggehally, M.~A.; Gunzburger, M.
{A forcebased blending model for atomistic-to-continuum coupling}.
{\it Inter. J. Multiscale Comput. Engineering.} 5(5), (2007), 387--406.

\bibitem[Badia et al., 2008]{Peri3}
Badia, S.; Parks, M.; Bochev, P.; Gunzburger, M.; Lehoucq, R.
{On atomistic-to-continuum coupling by blending}.
{\it Multiscale Modeling Simulation.} 7(1), (2008), 381--406.

\bibitem[Badiale and Serra, 2011]{badiale}
Badiale, M.; Serra, E.
\emph{Semilinear Elliptic Equations for Beginners: Existence Results via the Variational Approach}.
Springer, 2011.

\bibitem[Bates and Chmaj, 1999]{BatesChmaj99}
Bates, P.; Chmaj, A.
An integrodifferential model for phase transitions: stationary solutions in higher dimensions.
\emph{J. Statist. Phys.} 95 (1999), no.~5--6, 1119--1139.

\bibitem[Bates et al., 1997]{BatesFifeRenWang97}
Bates, P.; Fife, P.; Ren, X.; Wang, X.
Traveling waves in a convolution model for phase transitions.
\emph{Arch. Ration. Mech. Anal.} 138 (1997), no.~2, 105--136.

\bibitem[Berestycki et al., 2015]{Bere}
Berestycki, H.; Coulon, A.-Ch.; Roquejoffre, J.-M.; Rossi, L.
{The effect of a line with nonlocal diffusion on Fisher-KPP propagation}.
{\it Math. Models Methods Appl. Sciences.} 25(13), (2015), 2519--2562.

\bibitem[Br\"andle et al., 2012]{BrandleChasseigneQuiros12}
Br\"andle, C.; Chasseigne, E.; Quir\'os, F.
Phase transitions with midrange interactions: a nonlocal Stefan model.
\emph{SIAM J. Math. Anal.} 44 (2012), no.~4, 3071--3100.

\bibitem[Ca\~nizo and Molino, 2018]{CanizoMolino18}
Ca\~nizo, J.~A.; Molino, A.
Improved energy methods for nonlocal diffusion problems.
\emph{Discrete Contin. Dyn. Syst.} 38 (2018), no.~3, 1405--1425.

\bibitem[Capanna et al., 2021]{monia}
Capanna, M.; Nakasato, J.; Pereira, M. C.; Rossi, J. D.
Homogenization for nonlocal problems with smooth kernels.
\emph{J. Dyn. Diff. Equat.} 41(6), (2021), 2777--2808.

\bibitem[Carrillo and Fife, 2005]{CarrilloFife05}
Carrillo, C.; Fife, P.
Spatial effects in discrete generation population models.
\emph{J. Math. Biol.} 50(2), (2005), 161--188.

\bibitem[Chasseigne et al., 2006]{ChasseigneChavesRossi06}
Chasseigne, E.; Chaves, M.; Rossi, J.~D.
Asymptotic behavior for nonlocal diffusion equations.
\emph{J. Math. Pures Appl.} 86(3), (2006), 271--291.

\bibitem[Chasseigne and Jakobsen, 2017]{chasseignejakobsen}
Chasseigne, E.; Jakobsen, E.~R.
On nonlocal quasilinear equations and their local limits.
\emph{Journal of Differential Equations} 262(5), (2017), 3759--3804.

\bibitem[Cort\'azar et al., 2007]{CortazarElguetaRossiWolanski07}
Cort\'azar, C.; Elgueta, M.; Rossi, J.~D.; Wolanski, N.
Boundary fluxes for nonlocal diffusion.
\emph{J. Differential Equations} 234(2), (2007), 360--390.

\bibitem[Cort\'azar et al., 2008]{CortazarElguetaRossiWolanski08}
Cort\'azar, C.; Elgueta, M.; Rossi, J.~D.; Wolanski, N.
How to approximate the heat equation with Neumann boundary conditions by nonlocal diffusion problems.
\emph{Arch. Ration. Mech. Anal.} 187(1), (2008), 137--156.

\bibitem[Cort\'azar et al., 2009]{cortazar}
Cort\'azar, C.; Elgueta, M.; Rossi, J.~D.
Nonlocal diffusion problems that approximate the heat equation with Dirichlet boundary conditions.
\emph{Israel J. Math.} 170, (2009), 53--60.

\bibitem[Coville and Dupaigne, 2007]{Coville}
Coville, J.; Dupaigne, L.
{On a non-local equation arising in population dynamics}.
{\it Proc. Royal Soc. Edinburgh A.} 137(4), (2007), 727--755.


\bibitem[Crandall and Liggett, 1971]{CrandallLiggett71}
Crandall, M.~G.; Liggett, T.~M.
Generation of semigroups of nonlinear transformations on general Banach spaces.
\emph{Amer. J. Math.} 93 (1971), 265--298.

\bibitem[D'Elia et al., 2017]{DEliaDuGunzburgerLehoucq17}
D'Elia, M.; Du, Q.; Gunzburger, M.; Lehoucq, R.
Nonlocal convection--diffusion problems on bounded domains and finite-range jump processes.
\emph{Comput. Methods Appl. Math.} 17(4), (2017), 707--722.

\bibitem[D'Elia et al., 2016a]{DEliaPeregoBochevLittlewood16}
D'Elia, M.; Perego, M.; Bochev, P.; Littlewood, D.
A coupling strategy for nonlocal and local diffusion models with mixed volume constraints and boundary conditions.
\emph{Comput. Math. Appl.} 71(11), (2016), 2218--2230.

\bibitem[D'Elia et al., 2016b]{delia3}
D'Elia, M.; Ridzal, D.; Peterson, K.~J.; Bochev, P.; Shashkov, M.
\emph{Optimization-based mesh correction with volume and convexity constraints}.
\emph{J. Comput. Phys.} 313 (2016), 455--477.

\bibitem[D'Elia et al., 2022]{SUR}
D'Elia, M.; Li, X.; Seleson, P.; Tian, X.; Yu, Y.
{A review of Local-to-Nonlocal coupling methods in nonlocal diffusion and nonlocal mechanics}.
{\it Peridyn. Nonlocal Model.} 4 (2022), 1--50.

\bibitem[Dal Maso, 1993]{DalMaso}
Dal Maso, G.
\emph{An Introduction to $\Gamma$-Convergence}.
Birkh\"auser, Boston, 1993.

\bibitem[Di Paola et al., 2009]{DiPaola}
Di Paola, M.; Failla, G.; Zingales, M.
{\it Physically-based approach to the mechanics of strong non-local linear elasticity theory}.
\emph{Journal of Elasticity} 97 (2009), 103--130.


\bibitem[dos Santos et al., 2021]{bruna}
dos Santos, B.~C.; Oliva, S.~M.; Rossi, J.~D.
A local/nonlocal diffusion model.
\emph{Applicable Analysis.} 101 (2021), 5213--5246.

\bibitem[Du et al., 2018]{DuLiLuTian18}
Du, Q.; Li, X.; Lu, J.; Tian, X.
A quasi-nonlocal coupling method for nonlocal and local diffusion models.
\emph{SIAM J. Numer. Anal.} 56 (2018), 1386--1404.

\bibitem[Evans, 2010]{evans}
Evans, L.~C.
\emph{Partial Differential Equations}.
AMS, Providence, RI, 2010.

\bibitem[Fife, 2003]{Fife03}
Fife, P.
Some nonclassical trends in parabolic and parabolic-like evolutions.
In: \emph{Trends in Nonlinear Analysis}, 153--191.
Springer, Berlin, 2003.

\bibitem[Gal and Warma, 2017]{GalWarma17}
Gal, C.~G.; Warma, M.
Nonlocal transmission problems with fractional diffusion and boundary conditions on non-smooth interfaces.
\emph{Comm. Partial Differential Equations} 42 (2017), 579--625.

\bibitem[G\'arriz et al., 2020]{quiros}
G\'arriz, A.; Quir\'os, F.; Rossi, J.~D.
Coupling local and nonlocal evolution equations.
\emph{Calc. Var. Partial Differential Equations} 59 (2020), Paper No.~112.

\bibitem[Han and Lubineau, 2012]{Han}
Han, F.; Lubineau, G.
{Coupling of nonlocal and local continuum models by the Arlequin approach}.
\emph{Int. J. Numer. Methods Eng.} 89 (2012), 671--685.

\bibitem[Hutson et al., 2003]{Hutson}
Hutson, V.; Martinez, S.; Mischaikow, K.; Vickers, G.~T.
{The evolution of dispersal}.
\emph{J. Math. Biol.} 47 (2003), 483--517.

\bibitem[Kriventsov, 2015]{Kriventsov15}
Kriventsov, D.
Regularity for a local--nonlocal transmission problem.
\emph{Arch. Ration. Mech. Anal.} 217 (2015), 1103--1195.

\bibitem[Lunardi, 1995]{Lunardi}
Lunardi, A.
\emph{Analytic Semigroups and Optimal Regularity in Parabolic Problems}.
Birkh\"auser, Basel, 1995.

\bibitem[Molino and Rossi, 2016]{molinorossi}
Molino, A.; Rossi, J.~D.
Nonlocal diffusion problems that approximate a parabolic equation with spatial dependence.
\emph{Z. Angew. Math. Phys.} 67 (2016), Article 41.

\bibitem[Mengesha and Du, 2014]{MenDu}
Mengesha, T.; Du, Q.
{The bond-based peridynamic system with Dirichlet-type volume constraint}.
\emph{Proc. Roy. Soc. Edinburgh A} (2014), 161--186.


\bibitem[Ouhabaz, 2005]{ouhabaz}
Ouhabaz, E.~M.
\emph{Analysis of Heat Semigroups}.
Princeton University Press, 2005.

\bibitem[Pazy, 1983]{pazy}
Pazy, A.
\emph{Semigroups of Linear Operators and Applications to Partial Differential Equations}.
Springer, 1983.

\bibitem[Pereira, 2025]{toninho}
Pereira, A.~L.
A Note on the Formulation of the Neumann Boundary Condition for a Nonlocal Problem.
\emph{Preprint arXiv:2509.12345}, 2025.

\bibitem[Seleson et al., 2013a]{Sel}
Seleson, P.; Bochev, S.; Parks, M.
{A force-based coupling scheme for peridynamics and classical elasticity}.
\emph{Comput. Mater. Sci.} 66 (2013), 34--49.

\bibitem[Seleson and Gunzburger, 2010]{Sel2}
Seleson, P.; Gunzburger, M.
{Bridging methods for atomistic-to-continuum coupling and their implementation}.
\emph{Commun. Comput. Phys.} 7 (2010), 831.

\bibitem[Seleson et al., 2013b]{Sel3}
Seleson, P.; Gunzburger, M.; Parks, M.~L.
{Interface problems in nonlocal diffusion and sharp transitions between local and nonlocal domains}.
\emph{Comput. Methods Appl. Mech. Engrg.} 266 (2013), 185--204.


\end{thebibliography}
\end{document}